\documentclass[12pt]{amsart}
\usepackage{fullpage}
\usepackage[utf8]{inputenc}
\usepackage[english]{babel}
\usepackage{mathtools}
\usepackage{amssymb}
\usepackage{amsthm}
\usepackage{mathrsfs}
\usepackage{nicefrac}
\usepackage{upgreek}
\usepackage[T1]{fontenc}
\usepackage{graphicx}
\usepackage{booktabs}
\usepackage{graphics}
\usepackage{hyperref}

\usepackage{tikz}
        \usetikzlibrary{positioning}
				\usetikzlibrary{arrows.meta}
\usepackage{xargs}

\usepackage{tikz-cd}
\tikzcdset{%
    triple line/.code={\tikzset{%
        double equal sign distance, 
        double=\pgfkeysvalueof{/tikz/commutative diagrams/background color}}},
    quadruple line/.code={\tikzset{%
        double equal sign distance, 
        double=\pgfkeysvalueof{/tikz/commutative diagrams/background color}}},
    Rrightarrow/.code={\tikzcdset{triple line}\pgfsetarrows{tikzcd implies cap-tikzcd implies}},
    RRightarrow/.code={\tikzcdset{quadruple line}\pgfsetarrows{tikzcd implies cap-tikzcd implies}}
}

\newlength{\myline}
\newcommandx*{\triplearrow}[4][1=0, 2=1]{
  \draw[line width=\myline,double distance=3\myline,#3] #4;
  \draw[line width=\myline,shorten <=#1\myline,shorten >=#2\myline,#3] #4;
}
\newcommandx*{\quadarrow}[4][1=0, 2=2.5]{
  \draw[line width=\myline,double distance=5\myline,#3] #4;
  \draw[line width=\myline,double distance=\myline,shorten <=#1\myline,shorten >=#2\myline,#3] #4;
}

\usetikzlibrary{arrows.meta}
\tikzset{%
  >={Latex[width=2mm,length=2mm]},
            base/.style = {rectangle, rounded corners, draw=black,
                           minimum width=0.5cm, minimum height=0.25cm,
                           text centered, font=\tiny},
  activityStarts/.style = {base, fill=blue!30},
       startstop/.style = {base, fill=red!30},
    activityRuns/.style = {base, fill=green!30},
         process/.style = {base, minimum width=1cm, fill=orange!15,
                           font=\tiny},}

\usepackage[all]{xy}
\usepackage{graphics}
\usepackage{color}
\usepackage[T1]{fontenc}
\usepackage{enumitem}

\makeatletter
\newtheorem*{rep@theorem}{\rep@title}
\newcommand{\newreptheorem}[2]{%
\newenvironment{rep#1}[1]{%
 \def\rep@title{#2 \ref{##1}}%
 \begin{rep@theorem}}%
 {\end{rep@theorem}}}
\makeatother

\newcommand{\dl}{\langle\langle}
\newcommand{\dr}{\rangle\rangle}
\newcommand{\card}[1]{|#1|}

\newcommand{\C}[1]{{\mathcal #1}}

\newtheorem{Theorem}{Theorem}[section]
\newreptheorem{Theorem}{Theorem}
\newtheorem{mTheorem}{Theorem}
\newtheorem{Lemma}[Theorem]{Lemma}
\newreptheorem{Lemma}{Lemma}
\newtheorem{Claim}{Claim}[Theorem]
\newtheorem{Corollary}[Theorem]{Corollary}
\newreptheorem{Corollary}{Corollary}
\newtheorem{Proposition}[Theorem]{Proposition}
\newreptheorem{Proposition}{Proposition}
\newtheorem{Conjecture}[Theorem]{Conjecture}
\theoremstyle{definition}\newtheorem{mydef}[Theorem]{Definition}
\theoremstyle{remark}\newtheorem{remark}[Theorem]{Remark}
\theoremstyle{definition}
\makeatletter\@addtoreset{case}{example}\makeatother
\theoremstyle{definition}

\begin{document}

\title{Explicit strong boundedness for higher rank symplectic groups}

\author{Alexander A. Trost}
\address{University of Aberdeen}
\email{r01aat17@abdn.ac.uk}

\begin{abstract}
This paper gives an explicit argument to show strong boundedness for ${\rm Sp}_{2n}(R)$ for $R$ a ring of S-algebraic integers or a semi-local ring. This gives a quantitative version of the abstract result in the paper \cite{General_strong_bound}. The results presented further generalize older results regarding strong boundedness by Kedra, Libman and Martin \cite{KLM} and Morris \cite{MR2357719} from ${\rm SL}_n$ to ${\rm Sp}_{2n}$. Further, the presented results completely solve the question of the asymptotic of strong boundedness for ${\rm Sp}_{2n}(R)$ for $R$ semi-local case with an argument that immediately generalizes to all other split Chevalley groups. 
\end{abstract}

\maketitle

\section*{Introduction}

In our previous paper \cite{General_strong_bound}, we studied the diameter of word metrics given by finitely many conjugacy classes on arithmetic Chevalley groups $G(\Phi,R)$ for $R$ either a semi-local ring or a ring of S-algebraic integers. The main result was the following theorem: 

\begin{mTheorem}\cite[Theorem~3.1]{General_strong_bound}\label{general_thm} 
Let $\Phi$ be an irreducible root system of rank at least $2$ and let $R$ be a commutative ring with $1$. Additionally, let $G(\Phi,R)$ be boundedly generated by root elements and if $\Phi=C_2$ or $G_2$, then we further assume $(R:2R)<\infty.$
Then there is a constant $C(\Phi,R)\in\mathbb{N}$ such that for all finite, normally generating subset $S$ of $G(\Phi,R)$, it holds
\begin{equation*}
\|G(\Phi,R)\|_S\leq C(\Phi,R)|S|.
\end{equation*}
\end{mTheorem}

For a group $G$ and a natural number $k\in\mathbb{N},$ the worst possible diameter of all conjugation word norm $\|\cdot\|_S$ for $S$ a normally generating set with $|S|=k$, is denoted by $\Delta_k(G).$ Consider Section~\ref{definition} for a more precise definition. Using this terminology, Theorem~\ref{general_thm} shows that $\Delta_k(G(\Phi,R))$
has an upper bound proportional to $k$. But Theorem~\ref{general_thm} was proven abstractly by way of a model-theoretic compactness argument and consequently gives no information on the minimal possible choice for $C(\Phi,R).$ 

In this paper, we remedy this issue in two ways: First, we will give an explicit upper bound on the $\Delta_k({\rm Sp}_{2n}(R))$ for $R$ semi-local or a ring of S-algebraic integers and $n\geq 3$. In the case of semi-local rings, this upper bound is linear in the rank of $\Phi$ and in the case of algebraic integers, it is quadratic in the rank of $\Phi.$ Second, we will show that if one such $R$ has at least $k$ maximal ideals, then $\Delta_k({\rm Sp}_{2n}(R))\geq 2nk$ holds. Combining these upper and lower bounds settles the asymptotics of $\Delta_k({\rm Sp}_{2n}(R))$ for semi-local rings in $k,n$ and the number of maximal ideals:
\
\begin{mTheorem}\label{strong_bound_explicit_semi_local}
Let $R$ be a principal ideal domain with precisely $q\in\mathbb{N}$ many maximal ideals and let $n\geq 3$ and $k\in\mathbb{N}$ be given.
Then 
\begin{enumerate}
\item{$\Delta_k({\rm Sp}_{2n}(R))\leq 576(3n-2)\min\{q,5nk\}$,}
\item{$\Delta_k({\rm Sp}_{2n}(R))\geq 2nk$ for $k\leq q$ and}
\item{$\Delta_k({\rm Sp}_{2n}(R))\geq 2nq$ for $k\geq q+1.$}
\end{enumerate}
\end{mTheorem}

This theorem generalizes in a certain sense classical results by Liebeck and Lawther \cite{lawther1998diameter} stating that the conjugacy diameters of finite groups of Lie type are proportional to the rank of the underlying root system. On the other hand, combining the aforementioned upper and lower bounds on $\Delta_k({\rm Sp}_{2n}(R))$ for $R$ a ring of S-algebraic integers of class number $1$ gives a restriction on the possible asymptotic of $\Delta_k({\rm Sp}_{2n}(R))$:

\begin{mTheorem}\label{strong_bound_explicit_alg_integer}
Let $R$ be a ring of S-algebraic integers with class number one and let $n\geq 3$ and $k\in\mathbb{N}$ be given. Further set 
\begin{equation*}
\Delta(R):=
\begin{cases} 
 &\text{135, if }R\text{ is a quadratic imaginary ring of integers or }\mathbb{Z}\\
 &\text{12, if }R\text{ is neither of the above}
\end{cases}
\end{equation*}
Then
\begin{enumerate}
\item{$\Delta_k({\rm Sp}_{2n}(R))\leq 960(12n+\Delta(R))nk$ and}
\item{$2nk\leq\Delta_k({\rm Sp}_{2n}(R)).$} 
\end{enumerate}
\end{mTheorem}

While the calculations done in this paper to provide upper bounds on $\Delta_k({\rm Sp}_{2n}(R))$ are similar to those appearing in the paper \cite{KLM} by Kedra, Libman and Martin for ${\rm SL}_n(R),$ there are some modifications necessary due to the presence of two different root lengths in the root system $C_n.$ As a consequence, rather than using one Hessenberg form to simplify matrices as in \cite{KLM}, we need two distinct Hessenberg forms to accommodate these two root lengths. As seen in \cite{General_strong_bound} these two root lengths are a major problem for $C_2,$ where we pay a price of introducing additional powers of $2$ for passing back and forth between short and long roots. However, for $C_n$ for $n\geq 3,$ the presence of a root subsystem of $C_n$ isomorphic to $A_{n-1}$ and spanned by simple roots enables us to avoid this particular problem. While we still pay a price to pass between different root lengths, the situation is overall much nicer than for $C_2$.

Providing lower bounds on $\Delta_k({\rm Sp}_{2n}(R))$ on the other hand is done by way of considering the conjugacy diameter of long root elements in the groups ${\rm Sp}_{2n}(K)$ for various fields $K.$\\  

The paper is divided into five sections: In the first section, we introduce some notation and definitions used in the rest of the paper. In the second section, we address how to write root elements in normal subgroups of ${\rm Sp}_{2n}(R)$ for $n\geq 3$ generated by a single conjugacy class as certain bounded products of said conjugate. In the third section, we talk about how stable range conditions can be used to show that ${\rm Sp}_{2n}(R)$ for $R$ semi-local or a ring of S-algebraic integers is generated by a number of conjugates of root elements proportional to $n$ and show the first part of Theorem~\ref{strong_bound_explicit_semi_local}. In the fourth section, we use bounded generation results for S-arithmetic Chevalley groups to prove the first part of Theorem~\ref{strong_bound_explicit_alg_integer}. In the last section, we describe normal generating sets $S$ of ${\rm Sp}_{2n}(R)$ whose word norms $\|\cdot\|_S$ have large diameters to finish the proofs of Theorem~\ref{strong_bound_explicit_semi_local} and \ref{strong_bound_explicit_alg_integer}.

\section*{Acknowledgments}

I want to thank Benjamin Martin for his continued support and advice. 

\section{Definitions}\label{definition}

Let $G$ be a group and $S$ a finite subset of $G$, such that the conjugacy classes $C_G(S)$ of $S$ in $G$ generate $G.$ In this paper $\|\cdot\|_S:G\to\mathbb{N}_0$ denotes the word norm given by $C_G(S)$ in $G$. We further set
\begin{equation*}
B_S(k):=\{A\in G|\|A\|_S\leq k\}
\end{equation*}
for $k\in\mathbb{N}$ and $\|G\|_S={\rm diam}(\|\cdot\|_S)$ of $G$ as the minimal $N\in\mathbb{N}$, such that $B_S(N)=G$ or as $+\infty$ if there is no such $N$.
Further define for $k\in\mathbb{N}$ the invariant 
\begin{equation*}
\Delta_k(G):=\sup\{{\rm diam}(\|\cdot\|_S)|\ S\subset G\text{ with }\card{S}\leq k,\dl S\dr=G\}\in\mathbb{N}_0\cup\{\rm\infty\}
\end{equation*}
with $\Delta_k(G)$ defined as $-\infty$, if there is no normally generating set $S\subset G$ with $\card{S}\leq k.$ The group $G$ is called \textit{strongly bounded}, if $\Delta_k(G)$ is finite for all $k\in\mathbb{N}$. Also note $\Delta_k(G)\leq\Delta_{k+1}(G)$ for all $k\in\mathbb{N}$.

For $\Phi$ an irreducible root system and $R$ a commutative ring with $1,$ we will omit defining the simply-connected split Chevalley-De-Mazure group $G(\Phi,R)$ and the corresponding root elements $\varepsilon_{\alpha}(x)$ in this paper. We instead refer the interested reader to \cite[Section~2.1,2.2]{General_strong_bound} for a brief account of such definitions and to \cite{MR3616493} for further details regarding root elements. The \textit{elementary subgroup} $E(\Phi,R)$ (or $E(R)$ if $\Phi$ is clear from the context) is defined as the subgroup of $G(\Phi,R)$ generated by the elements $\varepsilon_{\phi}(x)$ for $\phi\in\Phi$ and $x\in R.$ It should be understood for the purposes of this paper that a system of simple roots in the root system $\Phi$ is chosen and fixed throughout. In particular, it is always clear which roots in $\Phi$ are positive and simple.

Further for a non-trivial ideal $I\subset R$, we denote the group homomorphism $G(\Phi,R)\to G(\Phi,R/I)$ induced by the quotient map $\pi_I:R\to R/I$ by $\pi_I$ as well. This group homomorphism is commonly called the \textit{reduction homomorphism induced by $I$.}

The subgroup $U^+(\Phi,R)$, called \textit{the subgroup of upper unipotent elements of $G(\Phi,R)$}, is the subgroup of $G(\Phi,R)$ generated by the root elements $\varepsilon_{\phi}(x)$ for $x\in R$ and $\phi\in\Phi$ a positive root. Similarly, one can define $U^-(\Phi,R)$, \textit{the subgroup of lower unipotent elements} of $G(\Phi,R)$ by root elements for negative roots.

Further, we define the following two word norms:

\begin{mydef}\label{root_elements_word_norms}
Let $R$ be a commutative ring with $1$ and $\Phi$ an irreducible root system such that $G(\Phi,R)$ is generated by root elements. Then define the two sets
\begin{align*}
{\rm EL}:=\{\varepsilon_{\phi}(t)|\ t\in R,\phi\in\Phi\}\text{ and }
{\rm EL}_Q:=\{A\varepsilon_{\phi}(t)A^{-1}|\ t\in R,\phi\in\Phi,A\in G(\Phi,R)\}.
\end{align*}
Then 
\begin{enumerate}
\item{define the word norm $\|\cdot\|_{{\rm EL}}:G(\Phi,R)\to\mathbb{N}_0$ as $\|1\|_{{\rm EL}}:=0$ and as
\begin{equation*}
\|X\|_{{\rm EL}}:=\min\{n\in\mathbb{N}|\exists A_1,\dots,A_n\in {\rm EL}: X=A_1\cdots A_n\}
\end{equation*}
for $X\neq 1.$}
\item{
define the word norm $\|\cdot\|_{{\rm EL}_Q}:G(\Phi,R)\to\mathbb{N}_0$ as $\|1\|_{{\rm EL}_Q}:=0$ and as
\begin{equation*}
\|X\|_{{\rm EL}_Q}:=\min\{n\in\mathbb{N}|\exists A_1,\dots,A_n\in {\rm EL}_Q: X=A_1\cdots A_n\}
\end{equation*}
for $X\neq 1.$}
\end{enumerate}
\end{mydef}

\begin{remark}
The group $G(\Phi,R)$ is \textit{boundedly generated by root elements}, if there is a natural number $N:=N(\Phi,R)\in\mathbb{N}$ such that 
$\|A\|_{{\rm EL}}\leq N$ holds for all $A\in G(\Phi,R).$
\end{remark}

The group elements $\varepsilon_{\phi}(t)$ are \textit{additive in $t\in R$}, that is $\varepsilon_{\phi}(t+s)=\varepsilon_{\phi}(t)\varepsilon_{\phi}(s)$ holds for all $t,s\in R$. Further, a couple of commutator formulas, expressed in the next lemma, hold. We will use the additivity and the commutator formulas implicitly throughout the thesis usually without reference.

\begin{Lemma}\cite[Proposition~33.2-33.5]{MR0396773}
\label{commutator_relations}
Let $R$ be a commutative ring with $1$ and let $\Phi$ be an irreducible root system of rank at least $2.$ Let $\alpha,\beta\in\Phi$ be roots with $\alpha+\beta\neq 0$ and let $a,b\in R$ be given.
\begin{enumerate}
\item{If $\alpha+\beta\notin\Phi$, then $(\varepsilon_{\alpha}(a),\varepsilon_{\beta}(b))=1.$}
\item{If $\alpha,\beta$ are positive, simple roots in a root subsystem of $\Phi$ isomorphic to $A_2$, then\\ 
$(\varepsilon_{\beta}(b),\varepsilon_{\alpha}(a))=\varepsilon_{\alpha+\beta}(\pm ab).$}
\item{If $\alpha,\beta$ are positive, simple roots in a root subsystem of $\Phi$ isomorphic to $C_2$ with $\alpha$ short and $\beta$ long, then
\begin{align*}
&(\varepsilon_{\alpha+\beta}(b),\varepsilon_{\alpha}(a))=\varepsilon_{2\alpha+\beta}(\pm 2ab)\text{ and}\\
&(\varepsilon_{\beta}(b),\varepsilon_{\alpha}(a))=\varepsilon_{\alpha+\beta}(\pm ab)\varepsilon_{2\alpha+\beta}(\pm a^2b).
\end{align*}
}
\end{enumerate}
\end{Lemma}

Before continuing, we will define the Weyl group elements in $G(\Phi,R)$:

\begin{mydef}
Let $R$ be a commutative ring with $1$ and let $\Phi$ be a root system. Define for $t\in R^*$ and $\phi\in\Phi$ the elements:
\begin{equation*}
w_{\phi}(t):=\varepsilon_{\phi}(t)\varepsilon_{-\phi}(-t^{-1})\varepsilon_{\phi}(t).
\end{equation*}
We will often write $w_{\phi}:=w_{\phi}(1).$
\end{mydef} 

Using these Weyl group elements, we obtain:

\begin{Lemma}\cite[Chapter~3, p.~23, Lemma~20(b)]{MR3616493}
\label{Weyl_group_conjugation_invariance1}
Let $R$ be a commutative ring with $1$ and $\Phi$ an irreducible root system with $\Pi$ its system of simple roots.
Let $\phi\in\Phi,\alpha\in\Pi$ and $x\in R,t\in R^*$ be given. Then $\varepsilon_{\phi}(x)^{w_{\alpha}}=\varepsilon_{w_{\alpha}(\phi)}(\pm x)$ holds and so for each $S\subset G(\Phi,R)$, one has 
\begin{equation*}
\|\varepsilon_{\phi}(x)\|_S=\|\varepsilon_{w_{\alpha}(\phi)}(x)\|_S.
\end{equation*}   
Here the element $w_{\alpha}(\phi)$ is defined by the action of $W(\Phi)$ on $\Phi$.
\end{Lemma}

In particular, Lemma~\ref{Weyl_group_conjugation_invariance1} implies for $\Phi$ an irreducible root system, $\phi\in\Phi, k\in\mathbb{N}$ and 
$S\subset G(\Phi,R)$, that the set $\{x\in R|\ \varepsilon_{\phi}(x)\in B_S(k)\}$ only depends on the length of the root $\phi$ and not on the particular $\phi$ in question. Thus the following definition makes sense:

\begin{mydef}
Let $R$ be a commutative ring with $1$ and $\Phi$ an irreducible root system and let $S\subset G(\Phi,R)$ be given. Then for $k\in\mathbb{N}_0$ 
define the subset $\varepsilon_s(S,k)$ of $R$ as $\{x\in R|\ \varepsilon_{\phi}(x)\in B_S(k)\}$ for any short root $\phi\in\Phi.$
\end{mydef}

Next, note:

\begin{mydef}
Let $R$ be a commutative ring with $1$, $I$ an ideal in $R$, $\Phi$ an irrducible root system and $S$ a subset of $G(\Phi,R)$. Then define the following two subsets of maximal ideals in $R:$
\begin{enumerate}
\item{$V(I):=\{m\text{ maximal ideal in }R|I\subset m\}$ and}
\item{$\Pi(S):=\{ m\text{ maximal ideal of $R$}|\ \forall A\in S:\pi_m(A)\text{ central in }G(\Phi,R/m)\}$}
\end{enumerate}
\end{mydef}

We also note the following observation:

\begin{Lemma}\label{intersection_v_Pi}
Let $R$ be a commutative ring with $1$, $I_1,I_2$ two ideals in $R,$ $\Phi$ an irreducible root system in $R$ and $S,T$ two subsets of $G(\Phi,R).$ Then
$V(I_1+I_2)=V(I_1)\cap V(I_2)$ and $\Pi(S\cup T)=\Pi(S)\cap\Pi(T)$ holds.
\end{Lemma}

The following corollary is crucial for the later analysis:

\begin{Corollary}\cite[Corollary~3.11]{General_strong_bound}
\label{necessary_cond_conj_gen}
Let $R$ be a commutative ring with $1$, $\Phi$ an irreducible root system of rank at least $3$ and assume $G(\Phi,R)=E(\Phi,R)$. 
Then a subset $S$ of $G(\Phi,R)$ normally generates $G(\Phi,R)$ precisely if $\Pi(S)=\emptyset.$
\end{Corollary}

\section{Generalized Hessenberg forms for ${\rm Sp}_{2n}(R)$ and level ideals}\label{section_matrix_calculations_sp_2n}

This section is quite similar to the proof of \cite[Theorem~6.1]{KLM}. The main problem with the following argument is not so much the actual argument, but the temptation to start the investigation with $n=2$ instead of $n\geq 3$. For this section, we use a representation of the complex, simply-connected Lie group ${\rm Sp}_{2n}(\mathbb{C})$ that gives the following, classical definition of $G(C_n,R)={\rm Sp}_{2n}(R):$

\begin{mydef}
Let $R$ be a commutative ring with $1$ and let 
\begin{equation*}
{\rm Sp}_{2n}(R):=\{A\in R^{2n\times 2n}|A^TJA=J\} 
\end{equation*}
be given with 
\begin{equation*}
J=
\left(\begin{array}{c|c}
0_n	& I_n \\
   \midrule
   -I_n & 0_n
\end{array}\right)
\end{equation*}
\end{mydef}

This implies the following:

\begin{Lemma}
Let $R$ be a commutative ring with $1$ and let $A\in Sp_{2n}(R)$ be given with
\begin{equation*}
A=\left(\begin{array}{c|c}
  A_1 & A_2 \\
   \midrule
   A_3 & A_4
\end{array}\right)
\end{equation*}
for $A_1,A_2,A_3,A_4\in R^{n\times n}.$ Then the equation 
\begin{equation*}
A^{-1}=-JA^TJ= 
\left(\begin{array}{c|c}
  A_1^T & -A_2^T \\
   \midrule
   -A_3^T & A_4^T
\end{array}\right)
\end{equation*}
holds. 
\end{Lemma}

We use this identity frequently in the following matrix calculations usually without reference. Every symplectic matrix can be writen as a $4\times 4$-block matrix of $n\times n$-matrices and this decomposition shows up naturally in the calculation. Therefore we will often signify this decomposition in blocks using vertical and horizontal lines in the following matrices as done in the above lemma for example. These lines serve merely as an optical help to read the calculations and have no further meaning. Let $n\geq 2$ be given. We can choose a system of positive simple roots $\{\alpha_1,\dots,\alpha_{n-1},\beta\}$ in $C_n$ such that the Dynkin-diagram of this system of positive simple roots has the following form 

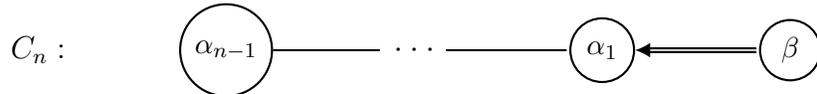
\begin{center}
				\begin{tikzpicture}[
        shorten >=1pt, auto, thick,
        node distance=2.5cm,
    main node/.style={circle,draw,font=\sffamily\small\bfseries},
		 mynode/.style={rectangle,fill=white,anchor=center}
														]
      \node[main node] (1) {$\beta$};
			\node[main node] (3) [left of=1] {$\alpha_1$};
			\node[mynode] (4) [left of=3] {$\cdot\cdot\cdot$};
			\node[main node] (5) [left of=4] {$\alpha_{n-1}$};
			\node[mynode] (6) [left of=5] {$C_n:$};
				\path (3) edge [double,<-] node {} (1);
				\path (4) edge [] node {} (3);
				\path (5) edge [] node {} (4);
						\end{tikzpicture}
						\end{center}	

Then subject to the choice of the maximal torus in ${\rm Sp}_{2n}(\mathbb{C})$ as diagonal matrices in ${\rm Sp}_{2n}(\mathbb{C})$, the root elements for simple roots in $G(C_n,R)={\rm Sp}_{2n}(R)$ can be chosen as: $\varepsilon_{\alpha_i}(t)=I_{2n}+t(e_{n-i,n-i+1}-e_{2n-i+1,2n-i})$ for $1\leq i\leq n-1$ and $\varepsilon_{\beta}(t)=I_{2n}+te_{n,2n}$ for all $t\in R.$ 

More generally, the root elements $\varepsilon_{\phi}(x)$ for short, positive roots in $\phi\in C_n$ and $x\in R$ are then either $I_{2n}+t(e_{ij}-e_{n+j,n+i})$ for $1\leq i<j\leq n$ or $I_{2n}+t(e_{i,n+j}+e_{j,n+i})$ for $1\leq i<j\leq n.$ The root elements $\varepsilon_{\psi}(x)$ for long, positive roots in $\psi\in C_n$ and $t\in R$ are then $I_{2n}+xe_{i,n+i}$ for $1\leq i\leq n$. Root elements for negative roots $\phi\in C_n$ and $x\in R$ are then $\varepsilon_{\phi}(x)=\varepsilon_{-\phi}(x)^T.$ The goal of this section is to prove the following:

\begin{Theorem}\label{level_ideal_explicit_Sp_2n}
Let $R$ be a principal ideal domain, $n\geq 3$ and let $A\in{\rm Sp}_{2n}(R)$ be given. Then there is an ideal $I(A)$ in $R$ such that
\begin{enumerate}
\item{$V(I(A))\subset\Pi(\{A\})$ and}
\item{$I(A)\subset\varepsilon_s(A,320n)$ hold.}
\end{enumerate}
\end{Theorem}

\subsubsection{The first Hessenberg form}

We start with a Lemma that gives us a conjugate of a matrix $A$ with a lot of zero entries similar to the Hessenberg forms used in \cite{KLM}:

\begin{Lemma}\label{first_Hessenberg_sp_2n}
Let $R$ be a principal ideal domain, $n\geq 3$ and $A\in Sp_{2n}(R)$ be given. Then there is an element 
$B\in Sp_{2n}(R)$ such that $A':=B^{-1}AB$ has the following form 
\begin{equation*}
A'=\left(\begin{array}{c|c}
  \begin{matrix} 
	a'_{1,1} & a'_{1,2} & a'_{1,3} & \cdot & a'_{1,n-2} & a'_{1,n-1} & a'_{1,n}\\
	a'_{2,1} & a'_{2,2} & a'_{2,3} & \cdot & a'_{2,n-2} & a'_{2,n-1} & a'_{2,n}\\
	0      & a'_{3,2} & a'_{3,3} & \cdot & a'_{3,n-2} & a'_{3,n-1} & a'_{3,n}\\
  0			 & 0			& a'_{4,3} & \cdot & a'_{4,n-2} & a'_{4,n-1} & a'_{4,n}\\
	\cdot  & \cdot  & \cdot  & \cdot & \cdot 		 & \cdot 		 & \cdot \\
	0			 & 0			& 0			 & \cdot & 0 & a'_{n,n-1} & a'_{n,n}		 
	\end{matrix}
	& A'_2 \\
   \midrule
   A'_3 & A'_4
\end{array}\right)
\end{equation*}
with $a'_{11}=a_{11}$ and $a'_{21}=gcd(a_{21},a_{31},\dots,a_{n1})$ up to multiplication with a unit in $R$ and $A'_2,A_3',A'_4\in R^{n\times n}$. We call a matrix of the form of $A'$ in $Sp_{2n}(R)$ a matrix in \textit{first Hessenberg form.}
\end{Lemma}

\begin{proof}
If $a_{3,1}=0$, then define $A^{(3)}:=A$.
Otherwise choose $t_3:=\gcd(a_{2,1},a_{3,1})$. Observe that $x_3:=-\frac{a_{3,1}}{t_3}$ and $y_3:=\frac{a_{2,1}}{t_3}$ are coprime elements of $R$ and hence, we can find elements $u_3,v_3\in R$ with $u_3y_3-x_3v_3=1.$ This implies that the matrix 
\begin{align*}
T_3:=
\left(\begin{array}{c|c}
\begin{matrix}
1 & \ & \ & \ \\
\ & u_3 & v_3 & \ \\
\ & x_3 & y_3 & \ \\
\ & \ & \ & I_{n-3}
\end{matrix}
& 0_n\\
\midrule
0_n & 
\begin{matrix}
1 & \ & \ & \ \\
\ & y_3 & -x_3 & \ \\
\ & -v_3 & u_3 & \ \\
\ & \ & \ & I_{n-3}
\end{matrix}
\end{array}
\right)
\end{align*}
is an element of ${\rm Sp}_{2n}(R).$ The matrix $A^{(3)}:=T_3AT_3^{-1}$ has the $(1,1)$-entry $a_{1,1}$ and the $(3,1)$-entry 
$x_3a_{2,1}+y_3a_{3,1}=-\frac{a_{3,1}}{t_3}a_{2,1}+\frac{a_{2,1}}{t_3}a_{3,1}=0.$ The entries of $A^{(3)}$ are denoted by $a_{k,l}^{(3)}.$ Next, if $a_{4,1}^{(3)}=0$, then define $A^{(4)}:=A^{(3)}.$ Otherwise choose $t_4:=\gcd(a_{2,1}^{(3)},a_{4,1}^{(3)})$. Observe that $x_4:=-\frac{a_{4,1}^{(3)}}{t_4}$ and $y_4:=\frac{a_{2,1}^{(3)}}{t_4}$ are coprime elements of $R$ and hence, we can find elements $u_4,v_4\in R$ with $u_4y_4-x_4v_4=1.$ This implies that the matrix 
\begin{align*}
T_4:=
\left(\begin{array}{c|c}
\begin{matrix}
1 & \ & \ & \ & \ \\
\ & u_4 & 0 & v_4 & \ \\
\ & 0 & 1 & 0 & \ \\
\ & x_4 & 0 & y_4 & \ \\
\ & \ & \ & \ & I_{n-4}
\end{matrix}
& 0_n\\
\midrule
0_n & 
\begin{matrix}
1 & \ & \ & \ & \ \\
\ & y_4 & 0 &-x_4 & \ \\
\ & 0 & 1 & 0 & \ \\
\ & -v_4 & 0 & u_4 & \ \\
\ & \ & \ & \ & I_{n-4}
\end{matrix}
\end{array}
\right)
\end{align*}
is an element of ${\rm Sp}_{2n}(R).$ The matrix $A^{(4)}:=T_4A^{(3)}T_4^{-1}$ has the $(1,1)$-entry $a_{1,1}^{(3)}=a_{1,1},$ the $(3,1)$-entry $0$ 
and the $(4,1)$-entry $x_4a_{2,1}^{(3)}+y_4a_{4,1}^{(3)}=-\frac{a_{4,1}^{(3)}}{t_4}a_{2,1}^{(3)}+\frac{a_{2,1}^{(3)}}{t_4}a_{4,1}^{(3)}=0.$ 
The entries of $A^{(4)}$ are denoted by $a_{k,l}^{(4)}.$ Carrying on this way, we find that the matrix $A^{(n)}$ is conjugate to $A$ in ${\rm Sp}_{2n}(R)$ and has the $(1,1)$-entry $a_{1,1}$ and $a_{3,1}^{(n)}=a_{4,1}^{(n)}=\cdots=a_{n,1}^{(n)}=0.$ Further, the construction implies the existence of a matrix $D\in{\rm SL}_{n-1}(R)$ with
\begin{equation*}
\left(\begin{array}{cc}
  1 & 
	\begin{matrix}
	0 & \cdots & 0
	\end{matrix}
	\\
	\begin{matrix}
	0\\ \cdot\\ \cdot\\ \cdot\\ 0
	\end{matrix}
	& D
\end{array}\right)\cdot 
\begin{pmatrix}
a_{1,1}\\ a_{2,1}\\ a_{3,1}\\ \cdot\\ \cdot\\ a_{n,1}
\end{pmatrix}
=\begin{pmatrix}
a_{1,1}\\ a_{2,1}^{(n)}\\ 0\\ \cdot\\ \cdot\\ 0
\end{pmatrix}
\end{equation*}
But this implies that $a_{2,1}^{(n)}$ is a multiple of $\gcd(a_{2,1},\dots,a_{n,1}).$ Further, note $D^{-1}\in{\rm SL}_{n-1}(R)$ and hence
\begin{equation*}
\left(\begin{array}{cc}
  1 & 
	\begin{matrix}
	0 & \cdots & 0
	\end{matrix}
	\\
	\begin{matrix}
	0\\ \cdot\\ \cdot\\ \cdot\\ 0
	\end{matrix}
	& D^{-1}
\end{array}\right)\cdot 
\begin{pmatrix}
a_{1,1}\\ a_{2,1}^{(n)}\\ 0\\ \cdot\\ \cdot\\ 0
\end{pmatrix}
=\begin{pmatrix}
a_{1,1}\\ a_{2,1}\\ a_{3,1}\\ \cdot\\ \cdot\\ a_{n,1}
\end{pmatrix}
\end{equation*} 
implies that all of the elements of $a_{2,1},\dots,a_{n,1}$ are multiples of $a_{2,1}^{(n)}$ and hence $\gcd(a_{2,1},\dots,a_{n,1})$ is also a multiple of 
$a_{2,1}^{(n)}.$ So, up to multiplication with a unit $a_{2,1}^{(n)}=\gcd(a_{2,1},\dots,a_{n,1}).$

Hence the first column of the matrix $A^{(n)}$ has the form described in the Lemma. The remaining columns of $A^{(n)}$ can be brought to the desired form in a similar way, by conjugating with a matrix of the form 
\begin{equation*}
\left(\begin{array}{c|c}
	\begin{matrix}
	I_2 & \ \\
	\ & D
	\end{matrix}
	& 0_n\\
\midrule
0_n & 
\begin{matrix}
	I_2 & \ \\
	\ & D^{-T}
	\end{matrix}
\end{array}\right)
\end{equation*}
for $D\in {\rm SL}_{n-2}(R).$ Note, that under conjugation with such a matrix, the first column of $A^{(n)}$ stays fixed and hence this yields the lemma. 
\end{proof}

\begin{remark}
\begin{enumerate}
\item{Upper Hessenberg matrices in $R^{n\times n}$ are matrices $A=(a_{ij})$ with $a_{ij}=0$ for $i>j+1.$ They are commonly used tools in numerical mathematics \cite{MR2978290} and define subvarieties of flag varieties which have been extensively studied \cite{MR1115324} as well.
}
\item{The proof strategy for Lemma~\ref{first_Hessenberg_sp_2n} is an adaption of \cite[Theorem~III.1]{MR0340283} to the group ${\rm Sp}_{2n}(R).$ Lemma~\ref{first_Hessenberg_sp_2n} (and Lemma~\ref{second_Hessenberg} describing the second Hessenberg form) are actually the only steps in the proof of Theorem~\ref{level_ideal_explicit_Sp_2n} requiring $R$ to be a principal ideal domain.}
\end{enumerate} 
\end{remark}

The strategy to prove Theorem~\ref{level_ideal_explicit_Sp_2n} is to calculate carefully chosen nested commuators of matrices in first (and in the next subsection second) Hessenberg-form with increasingly less entries until one arrives at root elements.

\begin{Lemma}
\label{first_commutator_formula_first_Hessenberg}
Let $R$ be a commutative ring with $1$ and $n\geq 3$ and let $A$ be a matrix in first Hessenberg form in $Sp_{2n}(R)$ and $B:=A^{-1}.$ Then 
$X:=(A,I_{2n}+e_{1,n+1})$ has the following form:
\begin{equation*}
X=\left(\begin{array}{c|c}
\begin{matrix} 
	x_{1,1} & x_{1,2} & \cdot & x_{1,n}\\
	x_{2,1} & x_{2,2} & \cdot & x_{2,n}\\
  0			 & 0			& \cdot & 0\\
	\cdot  & \cdot  & \cdot  & \cdot\\
	0      & 0 & \cdot & 1 \\ 
  \hline\
	x_{n+1,1} & x_{n+1,2} & \cdot & x_{n+1,n}\\ 
	\cdot        & \cdot  & \cdot  & \cdot\\ 
	x_{2n,1} & x_{2n,2} & \cdot & x_{2n,n}				 	 
	\end{matrix}	
&
\begin{matrix}
x_{1,n+1} & x_{1,n+2} & 0			 & \cdot 				 & 0\\
x_{2,n+1} & x_{2,n+2} & 0			 & \cdot 				 & 0\\
0 & 0 & 0 & \cdot & 0\\
\cdot 		 & \cdot & \cdot&\cdot&\cdot\\
0 & 0 & 0 & \cdot & 0\\
\hline\
x_{n+1,n+1} & x_{n+1,n+2} & 0			 & \cdot 				 & 0\\
\cdot 		 & \cdot & \cdot&\cdot&\cdot\\
x_{2n,n+1} & x_{2n,n+2} & 0			 & \cdot 				 & 1\\
\end{matrix}
\end{array}\right)
\end{equation*} 
with $x_{1,n+1}=a_{11}(b_{n+1,n+1}-b_{n+1,1})-1$ and $x_{2,n+1}=a_{21}(b_{n+1,n+1}-b_{n+1,1}).$ 
\end{Lemma}

\begin{proof}
Now let $A_2,A_3,A_4\in R^{n\times n}$ be given such that $A$ has the following form:
\begin{equation*}
A=\left(\begin{array}{c|c}
  \begin{matrix} 
	a_{11} & a_{12} & a_{13} & \cdot & a_{1,n-2} & a_{1,n-1} & a_{1n}\\
	a_{21} & a_{22} & a_{23} & \cdot & a_{2,n-2} & a_{2,n-1} & a_{2n}\\
	0      & a_{32} & a_{33} & \cdot & a_{3,n-2} & a_{3,n-1} & a_{3n}\\
  0			 & 0			& a_{43} & \cdot & a_{4,n-2} & a_{4,n-1} & a_{4n}\\
	\cdot  & \cdot  & \cdot  & \cdot & \cdot 		 & \cdot 		 & \cdot \\
	0			 & 0			& 0			 & \cdot & 0 & a_{n,n-1} & a_{n,n}		 
	\end{matrix}
	& A_2 \\
   \midrule
   A_3 & A_4
\end{array}\right)
\end{equation*}
Then for the matrices $B_2:=-A_2^T,B_3:=-A_3^T,B_1:=A_4^T\in R^{n\times n},$ one has:
\begin{equation*}
B=\left( \begin{array}{c|c}
   B_1 & B_2 \\
   \midrule
   B_3 & 
	\begin{matrix} 
	b_{n+1,n+1} & b_{n+1,n+2} & 0			 & \cdot &				0				 & 0 & 0\\
	\cdot  & \cdot  & \cdot  & \cdot & \cdot 		 & \cdot   & \cdot \\	
	b_{2n-3,n+1}			 & b_{2n-3,n+2}			& b_{2n-3,n+3}			 & \cdot & b_{2n-3,2n-2} & 0           & 0\\
	b_{2n-2,n+1}			 & b_{2n-2,n+2}			& b_{2n-2,n+3}			 & \cdot & b_{2n-2,2n-2} & b_{2n-2,2n-1}           & 0\\
	b_{2n-1,n+1}			 & b_{2n-1,n+2}			& b_{2n-1,n+3}			 & \cdot & b_{2n-1,2n-2} & b_{2n-1,2n-1}           & b_{2n-1,2n}\\
	b_{2n,n+1}			   & b_{2n,n+2}				& b_{2n,n+3}			 	 & \cdot   & b_{2n,2n-2}   & b_{2n,2n-1}             & b_{2n,2n}		 
	\end{matrix}
\end{array}\right)
\end{equation*}

Observe first:
\begin{align*}
&Ae_{1,n+1}A^{-1}=Ae_{1,n+1}B\\
&=\left(\begin{array}{c|c} 
  \begin{matrix}
	0 &\cdot & 0\\
	0 &\cdot & 0\\
	0 &\cdot & 0\\
  \cdot &\cdot & \cdot\\
	0 &\cdot & 0\\
	\hline\
0 &\cdot & 0\\
	\cdot & \cdot & \cdot\\
	0 &\cdot & 0\\
	\end{matrix}
	&
	\begin{matrix} 
	a_{1,1} & 0 & 0 & \cdot & 0 & 0 & 0\\
	a_{2,1} & 0 & 0 & \cdot & 0 & 0 & 0\\
	0      & 0 & 0 & \cdot & 0 & 0 & 0\\
	\cdot  & \cdot  & \cdot  & \cdot & \cdot 		 & \cdot & \cdot \\
	0			 & 0			& 0			 & \cdot & 0 & 0 & 0\\
	\hline\		
	a_{n+1,1}    & 0      & 0 		 & \cdot & 0 				 & 0     & 0\\
	\cdot        & \cdot  & \cdot  & \cdot & \cdot 		 & \cdot & \cdot\\
	a_{2n,1}		 & 0			& 0			 & \cdot & 0 				 & 0 		 & 0		 
	\end{matrix}	
\end{array}\right)
B
\end{align*}
\begin{align*}
&=
\left(\begin{array}{c|c} 
\begin{matrix} 
	a_{11}b_{n+1,1} & a_{11}b_{n+1,2} & \cdot & a_{11}b_{n+1,n}\\
	a_{21}b_{n+1,1} & a_{21}b_{n+1,2} & \cdot & a_{21}b_{n+1,n}\\ 
	0      & 0 & \cdot & 0\\ 
	\cdot  & \cdot  & \cdot  & \cdot\\
	0      & 0 & \cdot & 0\\ 
	\hline\
	a_{n+1,1}b_{n+1,1} & a_{n+1,1}b_{n+1,2} & \cdot & a_{n+1,1}b_{n+1,n}\\ 
	\cdot        & \cdot  & \cdot  & \cdot\\ 
	a_{2n,1}b_{n+1,1} & a_{2n,1}b_{n+1,2} & \cdot & a_{2n,1}b_{n+1,n}\\ 				 	 
	\end{matrix}
	&
	\begin{matrix}
	a_{11}b_{n+1,n+1} & a_{11}b_{n+1,n+2} & 0			 & \cdot 				 & 0\\
	a_{21}b_{n+1,n+1} & a_{21}b_{n+1,n+2} & 0			 & \cdot &				0				 \\
	0 & 0 & 0 & \cdot & 0 \\
	\cdot 	&\cdot	 & \cdot & \cdot & \cdot\\
	0 & 0 & 0 & \cdot & 0 \\
	\hline\
	a_{n+1,1}b_{n+1,n+1} & a_{n+1,1}b_{n+1,n+2} & 0			 & \cdot &				0				  \\
	\cdot 		 & \cdot & \cdot&\cdot&\cdot\\
	a_{2n,1}b_{n+1,n+1} & a_{2n,1}b_{n+1,n+2} & 0			 & \cdot &				0
	\end{matrix}
	\end{array}\right)
\end{align*}
This implies that
\begin{align*}
&(A,I_{2n}+e_{1,n+1})=A(I_{2n}+e_{1,n+1})A^{-1}(I_{2n}-e_{1,n+1})=(I_{2n}+Ae_{1,n+1}A^{-1})(I_{2n}-e_{1,n+1})\\
&=I_{2n}+Ae_{1,n+1}A^{-1}-e_{1,n+1}-Ae_{1,n+1}A^{-1}e_{1,n+1}\\
&\\
&=
\tiny\left(\begin{array}{c|c}
\begin{matrix} 
	1+a_{11}b_{n+1,1} & a_{11}b_{n+1,2} & \cdot & a_{11}b_{n+1,n}\\
	a_{21}b_{n+1,1} & 1+a_{21}b_{n+1,2} & \cdot & a_{21}b_{n+1,n}\\
  0			 & 0			& \cdot & 0\\
	\cdot  & \cdot  & \cdot  & \cdot\\
	0      & 0 & \cdot & 1 \\ 
  \hline\
	a_{n+1,1}b_{n+1,1} & a_{n+1,1}b_{n+1,2} & \cdot & a_{n+1,1}b_{n+1,n}\\ 
	\cdot        & \cdot  & \cdot  & \cdot\\ 
	a_{2n,1}b_{n+1,1} & a_{2n,1}b_{n+1,2} & \cdot & a_{2n,1}b_{n,n}\\ 				 	 
	\end{matrix}	
&
\begin{matrix}
a_{11}(b_{n+1,n+1}-b_{n+1,1})-1 & a_{11}b_{n+1,n+2} & 0			 & \cdot 				 & 0\\
a_{21}(b_{n+1,n+1}-b_{n+1,1}) & a_{21}b_{n+1,n+2} & 0			 & \cdot 				 & 0\\
0 & 0 & 0 & \cdot & 0\\
\cdot 		 & \cdot & \cdot&\cdot&\cdot\\
0 & 0 & 0 & \cdot & 0\\
\hline\
a_{n+1,1}(b_{n+1,n+1}-b_{n+1,1})+1 & a_{n+1,1}b_{n+1,n+2} & 0			 & \cdot &	0				  \\
	\cdot 		 & \cdot & \cdot&\cdot&\cdot\\
	a_{2n,1}(b_{n+1,n+1}-b_{n+1,1}) & a_{2n,1}b_{n+1,n+2} & 0			 & \cdot &				1
\end{matrix}
\end{array}\right)
\end{align*}
This is precisely the form claimed in the lemma.
\end{proof}

Next, we use the commutator from the previous lemma to obtain the following:

\begin{Lemma}\label{second_commutator_formula_first_Hessenberg}
Let $R$ be a commutative ring with $1$ and $n\geq 3$ and let $A$ be a matrix in first Hessenberg form in $Sp_{2n}(R)$ and $B=A^{-1}.$ Then
\begin{equation*}
(a_{11}(b_{n+1,n+1}-b_{n+1,1})-1,a_{21}(b_{n+1,n+1}-b_{n+1,1}))\subset\varepsilon_s(A,32).
\end{equation*}
\end{Lemma}

\begin{proof}
Let $X$ be the matrix obtained from $A$ as in Lemma~\ref{first_commutator_formula_first_Hessenberg} and $Y$ its inverse. We will prove the lemma by showing that
the two principal ideals $(a_{21}(b_{n+1,n+1}-b_{n+1,1}))$ and $(a_{11}(b_{n+1,n+1}-b_{n+1,1})-1)$ are both subsets of $\varepsilon_s(A,16).$ In order to show the first inclusion, we will show for $x\in R$ arbitrary that
\begin{align*}
B_A(16)\ni&(((X,I_{2n}+e_{2n,1}+e_{n+1,n}),I_{2n}+e_{n+1,1}),I_{2n}+x(e_{12}-e_{n+2,n+1}))\\
&=I_n-\left(a_{11}(b_{n+1,n+1}-b_{n+1,1})-1\right)x(e_{2n,2}+e_{n+2,n}).
\end{align*}
We first study the following term:
\begin{align*}
&X(e_{2n,1}+e_{n+1,n})X^{-1}=(Xe_{2n,1})X^{-1}+X(e_{n+1,n}X^{-1})=e_{2n,1}Y+Xe_{n+1,n}\\
&=\left(\begin{array}{c|c}
	\begin{matrix}
		0    & 0      & \cdot & 0\\
		\cdot  &\cdot &\cdot &\cdot\\
		0 & 0 & \cdot & 0\\
		0 & 0 & \cdot & 0\\
		\hline\
	0					& 0   & \cdot & 0\\
	\cdot & \cdot & \cdot & \cdot\\
	0					& 0   & \cdot & 0\\
	y_{11} & y_{12} & \cdot & y_{1n}  
	\end{matrix}
	&
	\begin{matrix}
	0					& 0         & 0 & \cdot & 0\\
	\cdot    & \cdot&\cdot &\cdot &\cdot\\ 
	0					& 0         & 0 & \cdot & 0\\
	0					& 0         & 0 & \cdot & 0\\
	\hline\
	0					& 0   & 0 & \cdot & 0\\
	\cdot    & \cdot&\cdot &\cdot &\cdot\\
	0					& 0         & 0 & \cdot & 0\\
	y_{1,n+1} & y_{1,n+2} & 0 & \cdot & 0
	\end{matrix}
	\end{array}\right)
	+
	\left(\begin{array}{c|c}
  \begin{matrix}
	0 & \cdot & 0 & x_{1,n+1}\\
	0 & \cdot & 0 & x_{2,n+1}\\
	0 & \cdot & 0 & 0\\
	\cdot & \cdot & \cdot & \cdot\\
	0 & \cdot & 0 & 0\\
	\hline\
	0 & \cdot & 0 & x_{n+1,n+1}\\
	\cdot & \cdot & \cdot & \cdot\\
	0 & \cdot & 0 & x_{2n,n+1}
	\end{matrix}
	&
	\begin{matrix}
	0 & \cdot & 0\\
	0 & \cdot & 0\\
	0 & \cdot & 0\\
	\cdot & \cdot & \cdot\\
	0 & \cdot & 0\\
	\hline\
	0 & \cdot & 0\\
	\cdot & \cdot & \cdot\\
	0 & \cdot & 0
	\end{matrix}
\end{array}\right)\\
&\\
&=
\left(\begin{array}{c|c}
  \begin{matrix}
	0 & \cdot & 0 & x_{1,n+1}\\
	0 & \cdot & 0 & x_{2,n+1}\\
	0 & \cdot & 0 & 0\\
	\cdot & \cdot & \cdot & \cdot\\
	0 & \cdot & 0 & 0\\
	\hline\
	0 & \cdot & 0 & x_{n+1,n+1}\\
	0 & \cdot & 0 & x_{n+2,n+1}\\
	\cdot & \cdot & \cdot & \cdot\\
	0 & \cdot & 0 & x_{2n-1,n+1}\\
	y_{11} & \cdot & y_{1,n-1} & y_{1n}+x_{2n,n+1}
	\end{matrix}
	&
	\begin{matrix}
	0 & 0 &0 & \cdot & 0\\
	0 & 0 &0 &\cdot & 0\\
	0 & 0 &0 &\cdot & 0\\
	\cdot &\cdot & \cdot & \cdot & \cdot\\
	0 & 0 &0 &\cdot & 0\\
	\hline\
	0 & 0 &0 &\cdot & 0\\
	0 & 0 &0 &\cdot & 0\\
	\cdot &\cdot & \cdot & \cdot & \cdot\\
	0 & 0 &0 &\cdot & 0 \\
	y_{1,n+1} & y_{1,n+2} &0 & \cdot & 0
	\end{matrix}
\end{array}\right)
\end{align*}
Next, observe that $X(e_{2n,1}+e_{n+1,n})X^{-1}(e_{2n,1}+e_{n+1,n})=y_{1,n+1}e_{2n,n}.$ Hence the innermost commutator $Z\in B_A(4)$ has the form: 
\begin{align*}
Z&=(X,I_{2n}+e_{2n,1}+e_{n+1,n})=(I_{2n}+X(e_{2n,1}+e_{n+1,n})X^{-1})(I_{2n}-e_{2n,1}-e_{n+1,n})\\
&=I_{2n}+X(e_{2n,1}+e_{n+1,n})X^{-1}-(e_{2n,1}+e_{n+1,n})
-X(e_{2n,1}+e_{n+1,n})X^{-1}(e_{2n,1}+e_{n+1,n})\\
&=I_{2n}+e_{2n,1}X^{-1}+Xe_{n+1,n}-e_{2n,1}-e_{n+1,n}-y_{1,n+1}e_{2n,n}
\end{align*}

Next, set $U:=Z^{-1}$. Then $U$ also has the form
\begin{align*}
\left(\begin{array}{c|c}
  \begin{matrix}
	1 & \cdot & 0 & u_{1,n}\\
	0 & \cdot & 0 & u_{2,n}\\
	0 & \cdot & 0 & 0\\
	\cdot & \cdot & \cdot & \cdot\\
	0 & \cdot & 0 & 1\\
	\hline\
	0 & \cdot & 0 & u_{n+1,n}\\
	0 & \cdot & 0 & u_{n+2,n}\\
	\cdot & \cdot & \cdot & \cdot\\
	0 & \cdot & 0 & u_{2n-1,n}\\
	u_{2n,1} & \cdot & u_{2n,n-1} & u_{2n,n}
	\end{matrix}
	&
	\begin{matrix}
	0 & 0 &0 & \cdot & 0\\
	0 & 0 &0 &\cdot & 0\\
	0 & 0 &0 &\cdot & 0\\
	\cdot &\cdot & \cdot & \cdot & \cdot\\
	0 & 0 &0 &\cdot & 0\\
	\hline\
	1 & 0 &0 &\cdot & 0\\
	0 & 1 &0 &\cdot & 0\\
	\cdot &\cdot & \cdot & \cdot & \cdot\\
	0 & 0 &0 &\cdot & 0 \\
	u_{2n,n+1} & u_{2n,n+2} &0 & \cdot & 1
	\end{matrix}
\end{array}\right)
\end{align*}
First, observe 
\begin{equation*}
Ze_{n+1,1}Z^{-1}=(e_{n+1,1}+z_{2n,n+1}e_{2n,1})U=e_{n+1,1}+u_{1,n}e_{n+1,n}+z_{2n,n+1}(e_{2n,1}+u_{1,n}e_{2n,n}).
\end{equation*}
This implies for the second inner-most commutator $S\in B_A(8)$ that
\begin{align*}
S:=(Z,I_{2n}+e_{n+1,1})&=(I_{2n}+Ze_{n+1,1}Z^{-1})(I_{2n}-e_{n+1,1})\\
&=(I_{2n}+e_{n+1,1}+u_{1,n}e_{n+1,n}+z_{2n,n+1}(e_{2n,1}+u_{1,n}e_{2n,n}))(I_{2n}-e_{n+1,1})\\
&=I_{2n}+u_{1,n}e_{n+1,n}+z_{2n,n+1}(e_{2n,1}+u_{1,n}e_{2n,n})
\end{align*}
with $u_{1,n}=z_{2n,n+1}=-z_{1,n}=-(a_{11}(b_{n+1,n+1}-b_{n+1,1})-1).$ But now one can easily check that in fact 
$(S,I_{2n}+x(e_{12}-e_{n+2,n+1}))=I_{2n}-\left(a_{11}(b_{n+1,n+1}-b_{n+1,1})-1\right)x(e_{2n,2}+e_{n+2,n})$ holds for all $x\in R$ and this nested commutator is an element of $B_A(16)$. This finishes the proof of the first inclusion. For the second inclusion, it suffices to show for $x\in R$ arbitrary that
\begin{align*}
B_A(16)\ni&(((X,I_{2n}+e_{2n,1}+e_{n+1,n}),I_{2n}+e_{n+2,2}),I_{2n}+x(e_{2,1}-e_{n+1,n+2}))\\
&=I_n-xa_{21}(b_{n+1,n+1}-b_{n+1,1})(e_{2n,1}+e_{n+1,n}).
\end{align*}
But this calculation works the same way as the one for the first inclusion, so we omit it.
\end{proof}

\subsubsection{The second Hessenberg Form}

\begin{Lemma}
\label{second_Hessenberg}
Let $R$ be a principal ideal domain and let $n\geq 3$ be given. Then for each $A\in{\rm Sp}_{2n}(R)$ there is a matrix $B\in{\rm Sp}_{2n}(R)$ such that 
$A':=BAB^{-1}$ has the form:
\begin{equation*}
A'=\left(\begin{array}{c|c}
  A'_1 & A'_2 \\
   \midrule
  \begin{matrix} 
	a'_{n+1,1} & a'_{n+1,2} & a'_{n+1,3} & \cdot & a'_{n+1,n-2} & a'_{n+1,n-1} & a'_{n+1,n}\\
	a'_{n+2,1} & a'_{n+2,2} & a'_{n+2,3} & \cdot & a'_{n+2,n-2} & a'_{n+2,n-1} & a'_{n+2,n}\\
	0      & a'_{n+3,2} & a'_{n+3,3} & \cdot & a'_{n+3,n-2} & a'_{n+3,n-1} & a'_{n+3,n}\\
  0			 & 0			& a'_{n+4,3} & \cdot & a'_{n+4,n-2} & a'_{n+4,n-1} & a'_{n+4,n}\\
	\cdot  & \cdot  & \cdot  & \cdot & \cdot 		 & \cdot 		 & \cdot \\
	0			 & 0			& 0			 & \cdot & 0 & a'_{2n,n-1} & a'_{2n,n}		 
	\end{matrix} 
	& A'_4
\end{array}\right)
\end{equation*}
with $a'_{n+2,1}=\gcd(a_{n+2,1},a_{n+3,1},\dots,a_{2n,1})$ up to a multiplication by a unit in $R.$ We call a matrix of the form of $A'$ in $Sp_{2n}(R)$ a matrix in \textit{second Hessenberg form.}
\end{Lemma}

We omit the proof, as it is very similar to the one of Lemma~\ref{first_Hessenberg_sp_2n}. One can then prove the following lemma by running through an analogous chain of calculations as the one showing Lemma~\ref{second_commutator_formula_first_Hessenberg}. 

\begin{Lemma}
\label{first_commutator_formula_second_Hessenberg}
Let $R$ be a commutative ring with $1$ and $n\geq 3$ and let $A$ be a matrix in second Hessenberg form in $Sp_{2n}(R)$ and $B=A^{-1}.$ Then 
\begin{equation*}
(a_{n+2,1}(b_{n+1,n+1}-b_{n+1,1}),a_{11}(b_{n+1,n+1}-b_{n+1,1})-1)\subset\varepsilon_s(A,32).
\end{equation*}
\end{Lemma}

\subsubsection{Constructing the level ideal}

We will apply the previous calculations to various matrices. First, note the following proposition:

\begin{Proposition}
\label{first_column}
Let $R$ be a principal ideal domain, $n\geq 3$ and $A=(a_{ij})_{1\leq i,j\leq 2n}\in{\rm Sp}_{2n}(R)$ be given. Then there are ideals
\begin{enumerate}
\item{$I_1^{(1)}(A)\subset\varepsilon_s(A,32)$ with $a_{2,1},\dots,a_{n,1}\in I_1^{(1)}(A)$ and}
\item{$I_1^{(2)}(A)\subset\varepsilon_s(A,32)$ with $a_{n+2,1},\dots,a_{2n,1}\in I_1^{(2)}(A)$.}
\end{enumerate}
We denote the ideal $I_1^{(1)}(A)+I_1^{(2)}(A)\subset\varepsilon_s(A,64)$ by $I_1(A).$
\end{Proposition}

\begin{proof}
The proof will be split in two parts. First we are going to construct the ideal $I_1^{(1)}(A)$ containing $a_{2,1},\dots,a_{n,1}$ and then the second ideal
$I_1^{(2)}(A)$ containing $a_{n+2,1},\dots,a_{2n,1}.$ For the first ideal put $A$ in first Hessenberg form and call the resulting matrix $A'=(a'_{ij})_{1\leq i,j\leq 2n}$ with inverse $B'=(b'_{ij})_{1\leq i,j\leq 2n}.$ Then apply Lemma~\ref{second_commutator_formula_first_Hessenberg} to $A'$ to obtain
\begin{equation*}
I^{(1)}_1(A):=(a'_{11}(b'_{n+1,n+1}-b'_{n+1,1})-1,a'_{21}(b'_{n+1,n+1}-b'_{n+1,1}))\subset\varepsilon_s(A,32).
\end{equation*}
Note $a'_{11}(b'_{n+1,n+1}-b'_{n+1,1})\equiv 1\text{ mod }I^{(1)}_1(A)$ and hence it follows 
\begin{equation*}
0=0\cdot a'_{11}\equiv a'_{21}(b'_{n+1,n+1}-b'_{n+1,1})a'_{11}\equiv a'_{21}\cdot 1=a'_{21}\text{ mod }I^{(1)}_1(A).
\end{equation*}
Thus $a'_{21}\in I^{(1)}_1(A)$ holds. But according to Lemma~\ref{first_Hessenberg_sp_2n}, the entry $a'_{21}$ of the matrix $A'$ agrees with ${\rm gcd}(a_{2,1},\dots,a_{n,1})$ up to multiplication with a unit for the entries $a_{21},\dots,a_{n,1}$ of the initial matrix $A$. So in particular, we obtain for an arbitrary matrix $A\in{\rm Sp}_{2n}(R)$ that $(a_{21},\dots,a_{n,1})$ is a subset of $I^{(1)}_1(A).$ Using the second Hessenberg form and Lemma~\ref{first_commutator_formula_second_Hessenberg} yields the ideal $I^{(2)}_1(A)\subset\varepsilon_s(A,32)$ with $a_{n+2,1},\dots,a_{2n,1}\in I^{(2)}_1(A)$. 
\end{proof}

The proposition yields all of-diagonal entries of the first column save for the single entry $a_{n+1,1}$ as arguments $x$ for root elements $\varepsilon_{\phi}(x)$
for $\phi\in C_n$ short. We can now prove Theorem~\ref{level_ideal_explicit_Sp_2n}:

\begin{proof}
First, define for $2\leq k\leq n$ the elements:
\begin{equation*}
w_k:=e_{1,k}-e_{k,1}+e_{n+1,n+k}-e_{n+k,n+1}+\sum_{1\leq j\leq 2n, j\neq 1,k,n+1,n+k} e_{j,j}\in{\rm Sp}_{2n}(R). 
\end{equation*}
The first column of the matrix $A_k:=w_kAw_k^{-1}$ is 
\begin{equation*}
\tiny{(a_{k,k},a_{2,k},\dots,a_{k-1,k},-a_{1,k},a_{k+1,k},\dots,a_{n,k},a_{n+k,k},a_{n+2,k}\dots,a_{n+k-1,k},-a_{n+1,k},a_{n+k+2,k},\dots,a_{2n,k})^T.}
\end{equation*}
Hence applying Proposition~\ref{first_column} to all of the matrices $A_2,\dots,A_n$ and the matrix $A_1:=A$, there are ideals
$I_1(A_1),\dots,I_1(A_n)$ all of them contained in $\varepsilon_s(A,64)$ with 
\begin{equation*}
a_{1,k},\dots,a_{n,k},a_{n+1,k},\dots,a_{n+k-1,k},a_{n+k+2,k},\dots,a_{2n,k}\in I_1(A_k)
\end{equation*}
for $k\geq 2$ and $a_{2,1},\dots,a_{n,1},a_{n+2,1},\dots,a_{2n,1}\in I_1(A_1).$ So, the ideal $I_2(A):=I_1(A_1)+\cdots+I_1(A_n)$ is contained in $\varepsilon_s(A,64n).$
Further, $I_2(A)$ contains all off-diagonal entries of the first $n$ columns of $A$ except possibly the entries $a_{n+1,1},a_{n+2,2},\dots,a_{2n,n}.$ Next, observe that $J$ itself is an element of ${\rm Sp}_{2n}(R)$ and choose $M_1,M_2,M_3,M_4\in R^{n\times n}$ with
\begin{equation*}
A=\left(\begin{array}{c|c}
	\begin{matrix}
		M_1\\
		\hline\
	M_3
	\end{matrix}
	&
	\begin{matrix}
	M_2\\
	\hline\
	M_4
	\end{matrix}
	\end{array}\right).
\end{equation*}

Then we obtain
\begin{align*}
A':=J^{-1}AJ&=
\left(\begin{array}{c|c}
	\begin{matrix}
		0_n\\
		\hline\
	I_n
	\end{matrix}
	&
	\begin{matrix}
	-I_n\\
	\hline\
	0_n
	\end{matrix}
	\end{array}\right)
	\cdot
	\left(\begin{array}{c|c}
	\begin{matrix}
		M_1\\
		\hline\
	M_3
	\end{matrix}
	&
	\begin{matrix}
	M_2\\
	\hline\
	M_4
	\end{matrix}
	\end{array}\right)
	\cdot J=
	\left(\begin{array}{c|c}
	\begin{matrix}
		-M_3\\
		\hline\
	M_1
	\end{matrix}
	&
	\begin{matrix}
	-M_4\\
	\hline\
	M_2
	\end{matrix}
	\end{array}\right)\cdot
	\left(\begin{array}{c|c}
	\begin{matrix}
		0_n\\
		\hline\
	-I_n
	\end{matrix}
	&
	\begin{matrix}
	I_n\\
	\hline\
	0_n
	\end{matrix}
	\end{array}\right)\\
	&\\
	&=\left(\begin{array}{c|c}
	\begin{matrix}
		M_4\\
		\hline\
	-M_2
	\end{matrix}
	&
	\begin{matrix}
	-M_3\\
	\hline\
	M_1
	\end{matrix}
	\end{array}\right)
\end{align*}
This implies, that if we apply the previous construction of $I_2(A)$ to the matrix $A'$, then we obtain an ideal 
$I_2(A')\subset\varepsilon_s(A',64n)=\varepsilon_s(A,64n)$ that contains all off-diagonal entries of the last $n$ columns of $A$, 
except possibly the entries $a_{1,n+1},\dots,a_{n,2n}.$ Thus if we consider the ideal $I'_3(A):=I_2(A)+I_2(A')\subset\varepsilon_s(A,128n)$, it follows:
\begin{align*}
A\equiv
\tiny{\left(\begin{array}{c|c}
	\begin{matrix}
		a_{11}      & 0 & 0 & \cdot &  0\\
		0      & a_{22} & 0 & \cdot &  0\\
		\cdot  &\cdot &\cdot &\cdot &\cdot \\
		0 & 0 & 0 & \cdot & a_{nn} \\ 
		\hline\
	a_{n+1,1}			& 0	& 0   & \cdot &0\\
	0			& a_{n+2,2}	& 0   & \cdot &0\\
	\cdot & \cdot & \cdot & \cdot & \cdot\\
	0 & 0 & 0 & \cdot & a_{2n,n} \\
	\end{matrix}
	&
	\begin{matrix}
	a_{1,n+1}			& 0	& 0   & \cdot &0\\
	0			& a_{2,n+2}	& 0   & \cdot &0\\
	\cdot & \cdot & \cdot & \cdot & \cdot\\
	0 & 0 & 0 & \cdot & a_{n,2n} \\
	\hline\
	a_{n+1,n+1}      & 0 & 0 & \cdot &  0\\
		0      & a_{n+2,n+2} & 0 & \cdot &  0\\
		\cdot  &\cdot &\cdot &\cdot &\cdot \\
		0 & 0 & 0 & \cdot & a_{2n,2n} \\ 
	\end{matrix}
	\end{array}\right)}
	 {\rm mod}\ I'_3(A).
\end{align*}
Thus the ideal $I_3(A):=(a_{ij},a_{i,n+j},a_{n+i,j},a_{n+i,n+j}|1\leq i\neq j\leq n)$ is contained in $I'_3(A)\subset\varepsilon_s(A,128n).$
Consequently, one also has
\begin{align*}
A^{-1}\equiv
\tiny{\left(\begin{array}{c|c}
	\begin{matrix}
		a_{n+1,n+1}      & 0 & 0 & \cdot &  0\\
		0      & a_{n+2,n+2} & 0 & \cdot &  0\\
		\cdot  &\cdot &\cdot &\cdot &\cdot \\
		0 & 0 & 0 & \cdot & a_{2n,2n} \\ 
		\hline\
	-a_{n+1,1}			& 0	& 0   & \cdot &0\\
	0			& -a_{n+2,2}	& 0   & \cdot &0\\
	\cdot & \cdot & \cdot & \cdot & \cdot\\
	0 & 0 & 0 & \cdot & -a_{2n,n} \\
	\end{matrix}
	&
	\begin{matrix}
	-a_{1,n+1}			& 0	& 0   & \cdot &0\\
	0			& -a_{2,n+2}	& 0   & \cdot &0\\
	\cdot & \cdot & \cdot & \cdot & \cdot\\
	0 & 0 & 0 & \cdot & -a_{n,2n} \\
	\hline\
	a_{1,1}      & 0 & 0 & \cdot &  0\\
		0      & a_{2,2} & 0 & \cdot &  0\\
		\cdot  &\cdot &\cdot &\cdot &\cdot \\
		0 & 0 & 0 & \cdot & a_{n,n} \\ 
	\end{matrix}
	\end{array}\right)}
	 {\rm mod}\ I_3(A).
\end{align*}

These congruences for $A$ and $A^{-1}$ imply
\begin{align*}
& A'':=(A,I_{2n}+e_{1,2}-e_{n+2,n+1})\\
&=\left(I_{2n}+A(e_{1,2}-e_{n+2,n+1})A^{-1}\right)\cdot(I_{2n}-e_{1,2}+e_{n+2,n+1})\\
&\equiv \left[I_{2n}+(a_{11}e_{12}+a_{n+1,1}e_{n+1,2}-a_{2,n+2}e_{2,n+1}-a_{n+2,n+2}e_{n+2,n+1})A^{-1}\right]\\
&\ \ \ \ \cdot (I_{2n}-e_{1,2}+e_{n+2,n+1})\\
&\equiv
[I_{2n}+a_{11}(a_{n+2,n+2}e_{12}-a_{2,n+2}e_{1,n+2})
+a_{n+1,1}(a_{n+2,n+2}e_{n+1,2}-a_{2,n+2}e_{n+1,n+2})\\
&\ \ \ \ -a_{2,n+2}(-a_{n+1,1}e_{2,1}+a_{11}e_{2,n+1})-a_{n+2,n+2}(-a_{n+1,1}e_{n+2,1}+a_{11}e_{n+2,n+1})]\\
&\ \ \ \ \cdot(I_{2n}-e_{1,2}+e_{n+2,n+1})\\
&=I_{2n}+a_{11}(a_{n+2,n+2}e_{12}-a_{2,n+2}e_{1,n+2})
+a_{n+1,1}(a_{n+2,n+2}e_{n+1,2}-a_{2,n+2}e_{n+1,n+2})\\
&\ \ \ \ -a_{2,n+2}(-a_{n+1,1}e_{2,1}+a_{11}e_{2,n+1})-a_{n+2,n+2}(-a_{n+1,1}e_{n+2,1}+a_{11}e_{n+2,n+1})\\
&\ \ \ \ -e_{1,2}+e_{n+2,n+1}-a_{n+1,1}a_{2,n+2}e_{22}-a_{n+1,1}a_{2,n+2}e_{n+1,n+1}\\
&\ \ \ \ -a_{n+2,n+2}a_{n+1,1}e_{n+2,2}-a_{11}a_{2,n+2}e_{1,n+1}\\
&=
\tiny\left(\begin{array}{c|c}
	\begin{matrix}
		1      & a_{11}a_{n+2,n+2}-1 & 0 & \cdot &  0\\
		a_{2,n+2}a_{n+1,1}  & 1-a_{n+1,1}a_{2,n+2} & 0 & \cdot &  0\\
		\cdot  &\cdot &\cdot &\cdot &\cdot \\
		0 & 0 & 0 & \cdot & 1 \\ 
		\hline\
	0			& a_{n+1,1}a_{n+2,n+2}	& 0   & \cdot &0\\
	a_{n+2,n+2}a_{n+1,1}			& -a_{n+2,n+2}a_{n+1,1}	& 0   & \cdot &0\\
	\cdot & \cdot & \cdot & \cdot & \cdot\\
	0 & 0 & 0 & \cdot & 0 \\
	\end{matrix}
	&
	\begin{matrix}
	-a_{11}a_{2,n+2}			& -a_{2,n+2}a_{11}	& 0   & \cdot &0\\
	-a_{2,n+2}a_{11}			& 0	& 0   & \cdot &0\\
	\cdot & \cdot & \cdot & \cdot & \cdot\\
	0 & 0 & 0 & \cdot & 0 \\
	\hline\
	1-a_{n+1,1}a_{2,n+2}      & -a_{n+1,1}a_{2,n+2} & 0 & \cdot &  0\\
		1-a_{n+2,n+2}a_{11}      & 1 & 0 & \cdot &  0\\
		\cdot  &\cdot &\cdot &\cdot &\cdot \\
		0 & 0 & 0 & \cdot & 1 \\ 
	\end{matrix}
	\end{array}\right)
	{\rm mod}\ I_3(A).
\end{align*}
Note that the $(n+2,1)$-entry $a''_{n+2,1}$ of $A''$ is congruent to $a_{n+2,n+2}a_{n+1,1}$ modulo $I_3(A)$ and the 
$(1,2)$-entry of $A''$ is congruent to $a_{n+2,n+2}a_{11}-1$ modulo $I_3(A).$ Further, note that $A''\in B_A(2).$ 
Next, apply Proposition~\ref{first_column}(2) to the matrix $A''$ to obtain an ideal 
\begin{equation*}
I_4^{(1)}(A):=I_1^{(3)}(A'')\subset\varepsilon_s(A'',32)\subset\varepsilon_s(A,64)
\end{equation*}
that contains $a''_{n+2,1}$, an element, which is congruent to $a_{n+2,n+2}a_{n+1,1}$ modulo $I_3(A).$

So for each element $X=(x_{ij})$ of ${\rm Sp}_{2n}(R)$, there is an ideal $I_4^{(1)}(X)\subset\varepsilon_s(X,64)$ which contains modulo $I_3(X)$ the element
$x_{n+2,n+2}x_{n+1,1}.$

Consider next, the matrix $B_A(2)\ni A''_{2}:=w_2A''w_2^{-1}$ and note that its $(2,1)$-entry is congruent modulo $I_3(A)$ to 
$a_{11}a_{n+2,n+2}-1.$ Apply Proposition~\ref{first_column}(1) to $A''_2$ to obtain an ideal
\begin{equation*}
I_4^{(2)}(A):=I_1^{(1)}(A''_2)\subset\varepsilon_s(A''_2,32)\subset\varepsilon_s(A,64)
\end{equation*}
that contains the $(2,1)$-entry of $A''_2$, which is congruent to $a_{11}a_{n+2,n+2}-1$ modulo $I_3(A).$
The properties of these ideals imply that the ideal $I^{(3)}_4(A):=I_4^{(1)}(A)+I_4^{(2)}(A)$ is contained in $\varepsilon_s(A,64+64)=\varepsilon_s(A,128)$ and
contains modulo $I_3(A),$ the elements $a_{n+2,n+2}a_{11}-1$ and $a_{n+2,n+2}a_{n+1,1}$ and consequently the element $a_{n+1,1}$ modulo $I_3(A).$ 

Phrased differently, for each matrix $X\in{\rm Sp}_{2n}(R)$, there is an ideal $I^{(3)}_4(X)\subset\varepsilon_s(X,128)$, which contains  the elements $x_{n+1,1}$ and $x_{n+2,n+2}x_{11}-1$ modulo the ideal $I_3(X).$

Observe that for $k=3,\dots,n$, the conjugate $A_k$ of $A$ defined before, has 
\begin{enumerate}
\item{$(n+1,1)$-entry equal to $a_{n+k,k}$,}
\item{$(n+2,n+2)$-entry equal to $a_{n+2,n+2}$ and}
\item{$(1,1)$-entry equal to $a_{k,k}$.}
\end{enumerate}
Further, the conjugate $A_2$ of $A$ defined before has
\begin{enumerate}
\item{$(n+1,1)$-entry equal to $a_{n+2,2}$,}
\item{$(n+2,n+2)$-entry equal to $a_{n+1,n+1}$ and}
\item{$(1,1)$-entry equal to $a_{2,2}$.}
\end{enumerate}
 
Hence applying the previous construction of the ideal $I^{(3)}_4(X)$ to the conjugates $A_2,A_3,\dots,A_n$ of $A$ then yields ideals 
$I^{(3)}_4(A_2),\dots,I^{(3)}_4(A_n)\subset\varepsilon_s(A,128)$
with the properties that 
\begin{enumerate}
\item{for $k=2,3\dots,n$, the ideal $I^{(3)}_4(A_k)$ contains the elements $a_{n+k,k}$ modulo the ideal $I_3(A_k)=I_3(A)$ and}
\item{for $k=3,\dots,n$, the ideal $I^{(3)}_4(A_k)$ contains the element $a_{n+2,n+2}a_{k,k}-1$ modulo the ideal $I_3(A_k)=I_3(A).$}
\end{enumerate}
To summarize, the ideal 
\begin{equation*}
I_4(A):=I_3(A)+I^{(3)}_4(A)+I^{(3)}_4(A_2)+\cdots+I^{(3)}_4(A_n)\subset\varepsilon_s(A,256n)
\end{equation*}
contains all the entries $a_{n+1,1},\dots,a_{2n,n}$ and $a_{n+2,n+2}a_{1,1}-1,a_{n+2,n+2}a_{3,3}-1,\dots,a_{n+2,n+2}a_{n,n}-1$.
This implies:
\begin{align*}
A\equiv
\tiny{\left(\begin{array}{c|c}
	\begin{matrix}
		a_{11}      & 0 & 0 & \cdot &  0\\
		0      & a_{22} & 0 & \cdot &  0\\
		\cdot  &\cdot &\cdot &\cdot &\cdot \\
		0 & 0 & 0 & \cdot & a_{nn} \\ 
		\hline\
	0			& 0	& 0   & \cdot &0\\
	0			& 0	& 0   & \cdot &0\\
	\cdot & \cdot & \cdot & \cdot & \cdot\\
	0 & 0 & 0 & \cdot & 0 \\
	\end{matrix}
	&
	\begin{matrix}
	a_{1,n+1}			& 0	& 0   & \cdot &0\\
	0			& a_{2,n+2}	& 0   & \cdot &0\\
	\cdot & \cdot & \cdot & \cdot & \cdot\\
	0 & 0 & 0 & \cdot & a_{n,2n} \\
	\hline\
	a_{n+1,n+1}      & 0 & 0 & \cdot &  0\\
		0      & a_{n+2,n+2} & 0 & \cdot &  0\\
		\cdot  &\cdot &\cdot &\cdot &\cdot \\
		0 & 0 & 0 & \cdot & a_{2n,2n} \\ 
	\end{matrix}
	\end{array}\right)}
	 {\rm mod}\ I_4(A).
\end{align*}
But $A$ is an element of ${\rm Sp}_{2n}(R)$ and hence $a_{ll}a_{n+l,n+l}\equiv 1\text{ mod }I_4(A)$ holds for all $l=1,\dots,n$. Thus $(a_{n+l,n+l}+I_4(A))^{-1}=a_{l,l}+I_4(A)$ holds in $R/I_4(A)$. On the other hand, $a_{n+2,n+2}a_{1,1}-1,a_{n+2,n+2}a_{3,3}-1,\dots,a_{n+2,n+2}a_{n,n}-1$ are all elements of $I_4(A)$ and hence 
\begin{equation*} 
a_{1,1}+I_4(A)=a_{3,3}+I_4(A)=\dots=a_{n,n}+I_4(A)=(a_{n+2,n+2}+I_4(A))^{-1}=a_{2,2}+I_4(A)
\end{equation*}
holds in the ring $R/I_4(A)$ as well. Thus we obtain
\begin{align*}
A\equiv
\tiny{\left(\begin{array}{c|c}
	\begin{matrix}
		a_{22}      & 0 & 0 & \cdot &  0\\
		0      & a_{22} & 0 & \cdot &  0\\
		\cdot  &\cdot &\cdot &\cdot &\cdot \\
		0 & 0 & 0 & \cdot & a_{22} \\ 
		\hline\
	0			& 0	& 0   & \cdot &0\\
	0			& 0	& 0   & \cdot &0\\
	\cdot & \cdot & \cdot & \cdot & \cdot\\
	0 & 0 & 0 & \cdot & 0 \\
	\end{matrix}
	&
	\begin{matrix}
	a_{1,n+1}			& 0	& 0   & \cdot &0\\
	0			& a_{2,n+2}	& 0   & \cdot &0\\
	\cdot & \cdot & \cdot & \cdot & \cdot\\
	0 & 0 & 0 & \cdot & a_{n,2n} \\
	\hline\
	a_{n+2,n+2}      & 0 & 0 & \cdot &  0\\
		0      & a_{n+2,n+2} & 0 & \cdot &  0\\
		\cdot  &\cdot &\cdot &\cdot &\cdot \\
		0 & 0 & 0 & \cdot & a_{n+2,n+2} \\ 
	\end{matrix}
	\end{array}\right)}
	 {\rm mod}\ I_4(A).
\end{align*}
Note in particular, that all diagonal entries of $A$ reduce to units in $R/I_4(A).$ 

Similarly, for $A'=J^{-1}AJ$ consider the conjugates $A'_k:=w_kA'w_k^{-1}$ for $k=2,\dots,n$. Observe that for $k=3,\dots,n$ the $(n+1,1)$-entry of $A'_k$ is 
$-a_{k,n+k}$ and the $(n+2,n+2)$-entry is $a_{2,2}.$ For $A'_2$ the 
$(n+1,1)$-entry is $-a_{2,n+2}$ and the $(n+2,n+2)$-entry is $a_{1,1}.$
Further, for $A'$ the $(n+1,1)$-entry is $-a_{1,n+1}$ and the $(n+2,n+2)$-entry is $a_{2,2}.$\\

Next, consider the ideals $I_4^{(1)}(A'),I_4^{(1)}(A'_2),\dots,I_4^{(1)}(A'_n)\subset\varepsilon_s(A,64)$ and observe that according to the construction of these ideals, one has that
\begin{enumerate}
\item{the ideal $I_4^{(1)}(A')$ contains the element $-a_{1,n+1}a_{2,2}$ modulo $I_3(A')=I_3(A),$}
\item{for $k=3,\dots,n$, the ideal $I_4^{(1)}(A'_k)$ contains the element $-a_{k,n+k}a_{2,2}$ modulo $I_3(A'_k)=I_3(A)$ and}
\item{the ideal $I_4^{(1)}(A'_2)$ contains the element $-a_{2,n+2}a_{1,1}$ modulo $I_3(A'_2)=I_3(A).$}
\end{enumerate} 

Next, consider the ideal:
\begin{align*}
I'(A)&:=I_3(A)+I^{(3)}_4(A)+I^{(3)}_4(A_2)+\cdots+I^{(3)}_4(A_n)+I_4^{(1)}(A')+I_4^{(1)}(A'_2)+\cdots+I_4^{(1)}(A'_n)\\
&\subset\varepsilon_s(A,256n+64n)=\varepsilon_s(A,320n).
\end{align*}
As $I_3(A)\subset I'(A)$, one concludes that 
\begin{enumerate}
\item{$-a_{1,n+1}a_{2,2}$ is an element of $I'(A),$}
\item{for $k=3,\dots,n$, the element $-a_{k,n+k}a_{2,2}$ is contained in $I'(A)$ and}
\item{the element $-a_{2,n+2}a_{1,1}$ is contained in $I'(A).$}
\end{enumerate} 
But remember that all diagonal entries of $A$ reduce to units in $R/I_4(A)$ and consequently also reduce to units in $R/I'(A).$ Hence as 
$a_{1,n+1}a_{2,2},a_{3,n+3}a_{2,2},\dots,a_{n,2n}a_{2,2}$ and $a_{2,n+2}a_{1,1}$ are all elements of $I'(A),$ we obtain that 
$a_{1,n+1},a_{3,n+3},\dots,a_{n,2n},a_{2,n+2}$ are also elements of $I'(A).$ Hence we obtain
\begin{align*}
A\equiv
\tiny{\left(\begin{array}{c|c}
	\begin{matrix}
		a_{22}      & 0 & 0 & \cdot &  0\\
		0      & a_{22} & 0 & \cdot &  0\\
		\cdot  &\cdot &\cdot &\cdot &\cdot \\
		0 & 0 & 0 & \cdot & a_{22} \\ 
		\hline\
	0			& 0	& 0   & \cdot &0\\
	0			& 0	& 0   & \cdot &0\\
	\cdot & \cdot & \cdot & \cdot & \cdot\\
	0 & 0 & 0 & \cdot & 0 \\
	\end{matrix}
	&
	\begin{matrix}
	0			& 0	& 0   & \cdot &0\\
	0			& 0	& 0   & \cdot &0\\
	\cdot & \cdot & \cdot & \cdot & \cdot\\
	0 & 0 & 0 & \cdot & 0 \\
	\hline\
	a_{n+2,n+2}      & 0 & 0 & \cdot &  0\\
		0      & a_{n+2,n+2} & 0 & \cdot &  0\\
		\cdot  &\cdot &\cdot &\cdot &\cdot \\
		0 & 0 & 0 & \cdot & a_{n+2,n+2} \\ 
	\end{matrix}
	\end{array}\right)}
	 {\rm mod}\ I'(A).
\end{align*}

Next, consider the ideal 
\begin{align*}
I(A)&:=I'_3(A)+I^{(3)}_4(A)+I^{(3)}_4(A_2)+\cdots+I^{(3)}_4(A_n)+I_4^{(1)}(A')+I_4^{(1)}(A'_2)+\cdots+I_4^{(1)}(A'_n)
\end{align*}
Remember that $I'_3(A)$ is also contained in $\varepsilon_s(A,128n)$ same as $I_3(A).$ Thus the ideal $I(A)$ is also contained in $\varepsilon_s(A,320n)$ same as 
$I'(A).$ Further, $I(A)$ contains $I'(A)$, because $I'_3(A)$ contains $I_3(A).$ Thus abusing notation and remembering 
$a_{11}\equiv a_{22}\text{ mod }I(A)$, we obtain 
\begin{equation*}
A\equiv a_{11}I_n\oplus a_{11}^{-1}I_n\text{ mod }I(A).
\end{equation*}
Next, remember that $I'_3(A)$ contains the ideal $I_1^{(1)}(A)$ and so according to the construction of $I_1^{(1)}(A)$ in the proof of Proposition~\ref{first_column},
the ideal $I'_3(A)$ contains the element $a'_{11}(b'_{n+1,n+1}-b'_{n+1,1})-1$ for $A'=(a'_{ij})$ being the matrix $A$ put in first Hessenberg-form and $B':=(A')^{-1}.$
However, we know from the proof of Lemma~\ref{first_Hessenberg_sp_2n} that $A'=DAD^{-1}$ for $D=D'\oplus (D')^{-T}$ for $D'\in{\rm SL}_n(R).$ Thus 
\begin{equation*}
A\equiv A'\text{ mod }I(A)\text{ and }B\equiv B'\text{ mod }I(A)
\end{equation*}
hold for $B:=A^{-1}.$ So, we conclude $a_{11}(b_{n+1,n+1}-b_{n+1,1})-1$ is an element of $I(A).$ However, $A$ is a diagonal matrix modulo $I(A)$ and so 
$B$ is as well. Thus $b_{n+1,1}$ is an element of $I(A)$ and further $b_{n+1,n+1}+I(A)=a_{11}+I(A)$ holds, too. So summarizing, we conclude that $a_{11}^2-1$ is an element of $I(A).$ 
To finish the proof let $m$ be an element of $V(I(A)).$ Then $(a_{11}-1)\cdot(a_{11}+1)=a_{11}^2-1$ is an element of $m$ and thus either 
\begin{equation*}
a_{11}\equiv 1\text{ mod }m\text{ or }a_{11}\equiv -1\text{ mod }m
\end{equation*}
holds. But in either case $a_{11}+m=(a_{11}+m)^{-1}$ holds and so $A$ reduces to a scalar matrix modulo $m.$ Thus $m\in\Pi(\{A\})$ and this finishes the proof.
\end{proof}

\begin{remark}
For a given element $A\in {\rm Sp}_{2n}(R),$ it is possible that any one of the many intermediate ideals $I$ making up $I(A)$ in the previous proof is already the entire ring $R.$ In this case, it is problematic to speak about units in the quotient $R/I$ or $R/I(A)$. However, if any of the intermediate ideals $I$ is already the entire ring $R$, then the claim of Theorem~\ref{level_ideal_explicit_Sp_2n} is obvious anyway, because then $V(I(A))=\emptyset$ holds. 
\end{remark}

We also note the following corollary:

\begin{Corollary}
\label{sum_decomposition_level_ideal_local}
Let $R$ be a principal ideal domain, $n\geq 3, A\in{\rm Sp}_{2n}(R)$. Then the ideal $I(A)$ of Theorem~\ref{level_ideal_explicit_Sp_2n}
is a sum of ideals $J_1(A),\dots,J_{7n}(A)$ such that $J_i(A)\subset\varepsilon_s(A,64)$ holds for all $1\leq i\leq 7n.$
\end{Corollary}

\begin{proof}
Recall the Weyl group elements  
\begin{equation*}
w_k:=e_{1,k}-e_{k,1}+e_{n+1,n+k}-e_{n+k,n+1}+\sum_{1\leq j\leq 2n, j\neq 1,k,n+1,n+k} e_{j,j}\in{\rm Sp}_{2n}(R). 
\end{equation*}
for $k=2,\dots,n.$
Then $X_k$ shall denote the conjugates $w_kXw_k^{-1}$ for $k=2,\dots,n$ and an arbitrary $X\in{\rm Sp}_{2n}(R).$
Going through the construction of $I(A)$ in the proof of Theorem~\ref{level_ideal_explicit_Sp_2n}, one can see that $I(A)$ is (contained in) the sum of the following ideals:
\begin{enumerate}
\item{
$I^{(1)}_1(A),I^{(1)}_1(A_2),\dots,I^{(1)}_1(A_n)$, $I^{(2)}_1(A),I^{(2)}_1(A_2),\dots,I^{(2)}_1(A_n)$,
$I^{(1)}_1(A'),I^{(1)}_1(A'_2),\dots,I^{(1)}_1(A'_n)$ and $I^{(2)}_1(A'),I^{(2)}_1(A'_2),\dots,I^{(2)}_1(A'_n)$ for $A':=J^{-1}AJ.$ These $4n$ ideals are all individually contained in $\varepsilon_s(A,32).$}
\item{$I_4^{(1)}(A),I_4^{(1)}(A_2),\dots,I_4^{(1)}(A_n)$ and $I_4^{(2)}(A),I_4^{(2)}(A_2),\dots,I_4^{(2)}(A_n).$ 
These $2n$ ideals are all individually contained in $\varepsilon_s(A,64).$}
\item{$I_4^{(1)}(A'),I_4^{(1)}(A'_2),\dots,I_4^{(1)}(A'_n).$ These $n$ ideals are all individually contained in $\varepsilon_s(A,64).$}
\end{enumerate}
So to summarize: $I(A)$ is the sum of $7n$ ideals that are all individually contained in $\varepsilon_s(A,64).$
\end{proof} 

\section{Stable range conditions, matrix decompositions and semi-local rings}

We first define the stable range of rings:

\begin{mydef}\cite[Ch.~1,§4]{MR0174604}
The \textit{(Bass) stable range} of a commutative ring $R$ with $1$ is the smallest $n\in\mathbb{N}$ with the following property:
If any $v_0,\dots,v_m\in R$ generate the unit ideal $R$ for $m\geq n$, then there are $t_1,\dots,t_m$ such that the elements 
$v_1':=v_1+t_1v_0,\dots,v_m':=v_m+t_mv_0$ also generate the unit ideal. If no such $n$ exists, $R$ has stable range $+\infty.$
\end{mydef}

Recalling the choices made for the symplectic group in Section~\ref{section_matrix_calculations_sp_2n}, we obtain the following decomposition for symplectic groups:

\begin{Proposition}
\label{Sp_2n_stable_upper_lower_decomposition}
Let $R$ be a ring of stable range at most $2$ such that the group ${\rm Sp}_4(R)$ is generated by its root elements and let $n\geq 2$ be given.
Then identifying ${\rm Sp}_4(R)$ with the subgroup
\begin{equation*}
{\rm Sp}_4(R)
=\left\{
\left(\begin{array}{c|c}
\begin{matrix}
I_{n-2} & \ \\
\ & A
\end{matrix}
&
\begin{matrix}
0_{n-2} & \ \\
\ & B
\end{matrix}\\
\midrule
\begin{matrix}
0_{n-2} & \ \\
\ & C
\end{matrix}
&
\begin{matrix}
I_{n-2} & \ \\
\ & D
\end{matrix}
\end{array}\right)
|\
\left(\begin{array}{c|c}
A & B\\
\midrule
C & D
\end{array} 
\right)\in{\rm Sp}_4(R)
\right\}
\end{equation*}
of ${\rm Sp}_{2n}(R)$, the following decomposition holds for the elementary subgroup $E(C_n,R)$ of ${\rm Sp}_{2n}(R):$
$E(C_n,R)=(U^+(C_n,R)\cdot U^-(C_n,R))^2\cdot{\rm Sp}_4(R).$
\end{Proposition} 

\begin{remark}
The product $(U^+(C_n,R)\cdot U^-(C_n,R))^2$ is a short hand for $\{A\cdot B\cdot C\cdot D|A,C\in U^+(C_n,R),B,D\in U^-(C_n,R)\}\subset{\rm Sp}_{2n}(R)$ and not a Cartesian product. 
\end{remark}

\begin{proof}
In Section~\ref{section_matrix_calculations_sp_2n}, we choose a system of positive simple roots $\{\alpha_1,\dots,\alpha_{n-1},\beta\}$ in $C_n$ such that the Dynkin-diagram of this system of positive simple roots has the following form 

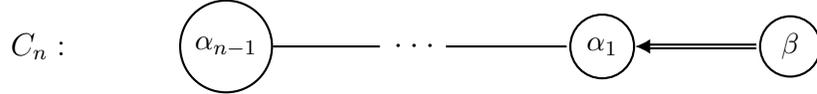
\begin{center}
				\begin{tikzpicture}[
        shorten >=1pt, auto, thick,
        node distance=2.5cm,
    main node/.style={circle,draw,font=\sffamily\small\bfseries},
		 mynode/.style={rectangle,fill=white,anchor=center}
														]
      \node[main node] (1) {$\beta$};
			\node[main node] (3) [left of=1] {$\alpha_1$};
			\node[mynode] (4) [left of=3] {$\cdot\cdot\cdot$};
			\node[main node] (5) [left of=4] {$\alpha_{n-1}$};
			\node[mynode] (6) [left of=5] {$C_n:$};
				\path (3) edge [double,<-] node {} (1);
				\path (4) edge [] node {} (3);
				\path (5) edge [] node {} (4);
						\end{tikzpicture}
						\end{center}	

We only sketch the proof of this proposition, which proceeds by induction on $n\in\mathbb{N}.$ First, the statement is obvious for $n=2.$ Next, set 
\begin{equation*}
X:=U^+(C_n,R)U^-(C_n,R)U^+(C_n,R)U^-(C_n,R){\rm Sp}_4(R).
\end{equation*}
Then $I_{2n}\in X$ holds and so quite similarly to classical proofs in algebraic K-theory like \cite[Theorem~2.5]{197877}, it suffices to show that for all simple roots $\phi\in C_n$ and $x\in R,$ one has $\varepsilon_{-\phi}(x)\cdot X\subset X.$ This in turn is done by distinguishing the cases $\phi=\alpha_{n-1}$ and 
$\phi\neq\alpha_{n-1}$ and arguing similarly as in the proof of Tavgen's \cite[Proposition~1]{MR1044049}. If $\phi\neq\alpha_{n-1}$, one uses the induction hypothesis and if $\phi=\alpha_{n-1},$ one uses a similar decomposition result for ${\rm SL}_n(R)$ by Vaserstein \cite[Lemma~9]{MR961333} for the subgroup $E(A_{n-1},R)$ given by the positive, simple roots $\alpha_1,\dots,\alpha_{n-1}$ in $C_n$.
\end{proof}

In a similar fashion to Proposition~\ref{Sp_2n_stable_upper_lower_decomposition}, one can prove the following proposition invoking \cite[Lemma~9]{MR961333} for the stable range $1$-case:

\begin{Proposition}
\label{rank_1_boundedness}
Let $R$ be a commutative ring with $1$ and $N\in\mathbb{N}$ such that 
\begin{align*}
&G(A_1,R)=E(A_1,R)=(U^+(A_1,R)U^-(A_1,R))^N,\\
&G(A_1,R)=E(A_1,R)=U^-(A_1,R)(U^+(A_1,R)U^-(A_1,R))^N
\end{align*}
or 
\begin{equation*}
G(A_1,R)=E(A_1,R)=(U^+(A_1,R)U^-(A_1,R))^NU^+(A_1,R)
\end{equation*}
holds.
Then 
\begin{align*}
&E(\Phi,R)=(U^+(\Phi,R)U^-(\Phi,R))^N,\\
&E(\Phi,R)=U^-(\Phi,R)(U^+(\Phi,R)U^-(\Phi,R))^N
\end{align*}
or 
\begin{equation*}
E(\Phi,R)=(U^+(\Phi,R)U^-(\Phi,R))^NU^+(\Phi,R)
\end{equation*}
respectively holds for all irreducible root systems $\Phi.$ Further, 
\begin{equation*}
E(\Phi,R)=(U^+(\Phi,R)U^-(\Phi,R))^2
\end{equation*} 
holds for $R$ a ring of stable range $1.$ 
\end{Proposition}

Next, we give a more detailed analysis of the asymptotics of bounded generation for ${\rm Sp}_{2n}.$ In this context, recall the word norm $\|\cdot\|_{{\rm EL}_Q}$ from Definition~\ref{root_elements_word_norms}.

\begin{Proposition}
\label{sp_2n_sl_n_stable_root_elements_principal_ideal_domain}
Let $R$ be a principal ideal domain and let $n\geq 3$. If ${\rm Sp}_4(R)$ and ${\rm Sp}_{2n}(R)$ are generated by its root elements and there is a $K\in\mathbb{N}$ with 
\begin{equation*}
\|{\rm Sp}_4(R)\|_{{\rm EL}_Q}\leq K,
\end{equation*}
then 
\begin{equation*}
\|{\rm Sp}_{2n}(R)\|_{{\rm EL}_Q}\leq 12(n-2)+K.
\end{equation*}
\end{Proposition}

\begin{proof}
Considering ${\rm Sp}_4(R)$ as a subgroup of ${\rm Sp}_{2n}(R)$ as done in Proposition~\ref{Sp_2n_stable_upper_lower_decomposition},
we first prove by induction that:

\begin{Claim}\label{claim_upper_triangular_elq} 
For each $A\in U^+(C_n,R)$ there is an $A'\in U^+(C_2,R)$ with $\|A'^{-1}A\|_{{\rm EL}_Q}\leq 3(n-2)$ for $n\geq 2.$ 
\end{Claim}

First, the claim is clear for $n=2$. Let $A\in U^+(C_n,R)$ be given. Then it has the form
\begin{align*}
A=
\left(\begin{array}{c|c}
\begin{matrix}
\begin{matrix}
1 & a_{1,2} & \cdot & \cdot & a_{1,n}\\
\ & 1 & a_{2,3} & \cdot & a_{2,n}\\
\ & \ & 1 & \cdot & \cdot \\
\ & \ & \ &  \cdot & \cdot \\
\ & \ & \ & \ & 1 \\
\end{matrix}\\
\midrule
\ \\
\ \\
O_n
\ \\
\ \\
\
\end{matrix}
& 
\begin{matrix}
a_{1,n+1} & \cdot & \cdot & \cdot & a_{1,2n}\\
a_{2,n+1} & \cdot & \cdot & \cdot & a_{2,2n}\\
a_{3,n+1} & \cdot & \cdot & \cdot & a_{3,2n}\\
\cdot & \cdot & \cdot & \cdot & \cdot \\
a_{n,n+1} & \cdot & \cdot & \cdot & a_{n,2n}\\
\midrule
\ 1 & \ & \ & \ & \ \\
-a_{1,2} & \ 1 & \ & \ & \ \\
\cdot & -a_{2,3} & 1 & \ & \ \\
\cdot & \cdot & \cdot & \cdot & \ \\
-a_{1,n} & -a_{2,n} & \cdot & \cdot & \ 1
\end{matrix}
\end{array}\right)
\end{align*}

Multiplying $A$ with the matrix
\begin{align*}
&T:=(I_{2n}-a_{1,2}(e_{1,2}-e_{n+2,n+1}))\cdot(I_{2n}-a_{1,3}(e_{1,3}-e_{n+3,n+1}))\cdots(I_{2n}-a_{1,n}(e_{1,n}-e_{2n,n+1}))
\end{align*}
from the right yields an element $B$ of $U^+(C_n,R)$ with the first $n$ entries of the first row of $B$ being $0$, except for the $(1,1)$-entry, which is $1.$
However, according to the proof of Lemma~\ref{first_Hessenberg_sp_2n}, there is a matrix $D\in{\rm Sp}_{2n}(R)$ of the form
\begin{align*}
D=
\left(
\begin{array}{c|c}
\begin{matrix}
1 & \ \\
\ & D'
\end{matrix}
& 0_n\\
\midrule
0_n &
\begin{matrix}
1 & \ \\
\ & D'^{-T}
\end{matrix} 
\end{array}
\right)
\end{align*}
for $D'\in{\rm SL}_{n-1}(R)$ such that the first column of $DT^TD^{-1}$ has the form
\begin{equation*}
(1,t,0,\dots,0)^T
\end{equation*}
for $t={\rm gcd}(-a_{1,2},-a_{1,3},\dots,-a_{1,n}).$ However, due to the form of $T^T$ and $D$, this implies that 
$DT^TD^{-1}=I_{2n}+t(e_{21}-e_{n+1,n+2})$ and hence $D^TTD^{-T}=I_{2n}+t(e_{12}-e_{n+2,n+1})$ holds. This implies $\|T\|_{{\rm EL}_Q}\leq 1.$ Then $B=A\cdot T$ has the form
\begin{align*}
B=\left(\begin{array}{c|c}
\begin{matrix}
\begin{matrix}
1 & 0 & \cdot & \cdot & 0\\
\ & 1 & b_{2,3} & \cdot & b_{2,n}\\
\ & \ & 1 & \cdot & \cdot \\
\ & \ & \ & \cdot & \cdot \\
\ & \ & \ & \ & 1 \\
\end{matrix}\\
\midrule
\ \\
\ \\
O_n
\ \\
\ \\
\
\end{matrix}
& 
\begin{matrix}
b_{1,n+1} & \cdot & \cdot & \cdot & b_{1,2n}\\
b_{2,n+1} & \cdot & \cdot & \cdot & b_{2,2n}\\
\cdot & \cdot & \cdot & \cdot & \cdot \\
b_{n,n+1} & \cdot & \cdot & \cdot & b_{n,2n}\\
b_{n,n+1} & \cdot & \cdot & \cdot & b_{n,2n}\\
\midrule
\ 1 & \ & \ & \ & \ \\
\ 0 & \ 1 & \ & \ & \ \\
\ 0 & -b_{2,3} & 1 & \ & \ \\
\ \cdot & \cdot & \cdot & \cdot & \ \\
\ 0 & -b_{2,n} & \cdot & \cdot & \ 1
\end{matrix}
\end{array}\right)
\end{align*}

Next, multiplying $B$ with
\begin{align*}
S:=&(I_{2n}-b_{1,n+1}e_{1,n+1})\cdot(I_{2n}-b_{1,n+2}(e_{1,n+2}+e_{2,n+1}))\cdot(I_{2n}-b_{1,n+3}(e_{1,n+3}+e_{3,n+1}))\\
&\cdots(I_{2n}-b_{1,2n}(e_{1,2n}+e_{n,n+1}))
\end{align*}
from the right yields an element $C\in U^+(C_n,R)$ whose first row is $(1,0,\dots,0).$ But applying the proof of Lemma~\ref{second_Hessenberg}, we can find a matrix of the form 
\begin{align*}
E=
\left(
\begin{array}{c|c}
\begin{matrix}
1 & \ \\
\ & E'
\end{matrix}
& 0_n\\
\midrule
0_n &
\begin{matrix}
1 & \ \\
\ & E'^{-T}
\end{matrix} 
\end{array}
\right)
\end{align*}
for $E'\in{\rm SL}_{n-1}(R)$ such that the first column of $ES^TE^{-1}$ has the form $(1,0,\dots,0,-b_{1,n+1},s,\dots,0)^T$
for $s={\rm gcd}(b_{1,n+2},b_{1,n+3},\dots,b_{1,2n}).$ However, due to the form of $S^T$ and $E$, this implies that 
$ES^TE^{-1}=(I_{2n}-b_{1,n+1}e_{n+1,1})\cdot(I_{2n}+s(e_{n+1,2}+e_{n+2,1}))$ and hence 
\begin{equation*}
E^TSE^{-T}=(I_{2n}-b_{1,n+1}e_{1,n+1})\cdot(I_{2n}+s(e_{1,n+2}+e_{2,n+1}))
\end{equation*}
holds. This implies that $\|T\|_{{\rm EL}_Q}\leq 2.$ 

But note that $C$ must be an element of the subgroup $U^+(C_{n-1},R)$ of $U^+(C_n,R)$, if its first row is $(1,0,\dots,0)^T.$
This yields by induction that there is a $C'\in U^+(C_2,R)$ with  
\begin{equation*}
\|C'^{-1}C\|_{{\rm EL}_Q}\leq 3(n-1-2)=3(n-3)
\end{equation*}
holds. Hence setting $A'$ as $C'$, one obtains from $C=ATS$ that
\begin{align*}
\|A'^{-1}A\|_{{\rm EL}_Q}&=\|C'^{-1}CS^{-1}T^{-1}\|_{{\rm EL}_Q}\leq\|C'^{-1}C\|_{{\rm EL}_Q}+\|T\|_{{\rm EL}_Q}+\|S\|_{{\rm EL}_Q}\\
&\leq 3(n-3)+3=3(n-2).
\end{align*}

Thus the claim holds for all $n\geq 2.$ Let $A\in{\rm Sp}_{2n}(R)$ be given. Principal ideal domains have stable range at most $2$ and so Proposition~\ref{Sp_2n_stable_upper_lower_decomposition} yields that 
\begin{equation*}
{\rm Sp}_{2n}(R)=(U^+(C_n,R)U^-(C_n,R))^2{\rm Sp}_4(R)
\end{equation*}
for all $n\geq 2.$ Hence there are $u_1^+,u_2^+\in U^+(C_n,R),u_1^-,u_2^-\in U^-(C_n,R)$ as well as 
$Z\in{\rm Sp}_4(R)$ with $A=u_1^+u_1^-u_2^+u_2^-Z$.
But $U^+(C_n,R)$ and $U^-(C_n,R)$ are conjugate in ${\rm Sp}_{2n}(R)$. Hence applying the claim of the first part of the proof to the $u_1^+,u_1^-,u_2^+,u_2^-$ yields 
$X_1,X_2,Y_1,Y_2\in{\rm Sp}_{2n}(R)$ with 
\begin{equation*}
\|X_1\|_{{\rm EL}_Q},\|X_2\|_{{\rm EL}_Q},\|Y_1\|_{{\rm EL}_Q},\|Y_2\|_{{\rm EL}_Q}\leq 3(n-2)
\end{equation*}
and $v_1^+,v_2^+\in U^+(C_2,R)$ and $v_1^-,v_2^-\in U^-(C_2,R)$ such that $u_1^+=v_1^+X_1,u_2^+=v_2^+X_2,u_1^-=v_1^-Y_1,u_2^-=v_2^-Y_2.$ 
But this implies 
\begin{align*}
A&=u_1^+u_1^-u_2^+u_2^-Z=(v_1^+X_1)\cdot(v_1^-Y_1)\cdot(v_2^+X_2)\cdot(v_2^-Y_2)Z\\
&=(v_1^+X_1(v_1^+)^{-1})\cdot(v_1^+v_2^-Y_1(v_1^+v_2^-)^{-1})\cdot(v_1^+v_2^-v_2^+X_2(v_1^+v_2^-v_2^+)^{-1})\\
&\ \ \ \cdot(v_1^+v_2^-v_2^+v_2^-X_2(v_1^+v_2^-v_2^+v_2^-)^{-1})\cdot(v_1^+v_2^-v_2^+v_2^-)\cdot Z\\
&=(X_1^{v_1^+})\cdot(Y_1^{v_1^+v_2^-})\cdot(X_2^{v_1^+v_2^-v_2^+})\cdot(Y_2^{v_1^+v_2^-v_2^+v_2^-})\cdot(v_1^+v_2^-v_2^+v_2^-)\cdot Z.
\end{align*}
But $(v_1^+v_2^-v_2^+v_2^-)\cdot Z$ is an element of ${\rm Sp}_4(R)$ and hence $\|(v_1^+v_2^-v_2^+v_2^-)\cdot Z\|_{{\rm EL}_Q}\leq K$ holds. This implies
\begin{align*}
\|A\|_{{\rm EL}_Q}
&=\|(X_1^{v_1^+})\cdot(Y_1^{v_1^+v_2^-})\cdot(X_2^{v_1^+v_2^-v_2^+})\cdot(Y_2^{v_1^+v_2^-v_2^+v_2^-})\cdot(v_1^+v_2^-v_2^+v_2^-)\cdot Z\|_{{\rm EL}_Q}\\
&\leq\|X_1\|_{{\rm EL}_Q}+\|Y_1\|_{{\rm EL}_Q}+\|X_2\|_{{\rm EL}_Q}+\|Y_2\|_{{\rm EL}_Q}+\|(v_1^+v_2^-v_2^+v_2^-)\cdot Z\|_{{\rm EL}_Q}\\
&\leq 4\cdot 3\cdot(n-2)+K=12(n-2)+K.
\end{align*}
This yields the statement of the proposition.
\end{proof}

One also obtains:

\begin{Corollary}\label{sp_2n_sl_n_stable_root_elements_principal_ideal_domain}
Let $R$ be a principal ideal domain of stable range $1$ with ${\rm Sp}_{2n}(R)=E(C_n,R)$ for $n\geq 2.$ Then 
$\|{\rm Sp}_{2n}(R)\|_{{\rm EL}_Q}\leq 9n-6.$ 
\end{Corollary}

\begin{proof}
First, note that according to Proposition~\ref{rank_1_boundedness}, we have ${\rm Sp}_{2n}(R)=(U^+(C_n,R)U^-(C_n,R))^2.$ Thus for each $A\in{\rm Sp}_{2n}(R),$ there are $u_1^+,u_2^+\in U^+(C_n,R)$ and $u_1^-,u_2^-\in U^-(C_n,R)$ with $A=u_1^+u_1^-u_2^+u_2^-.$ Then 
\begin{align*}
\|A\|_{{\rm EL}_Q}&=\|u_1^+u_1^-u_2^+u_2^-\|_{{\rm EL}_Q}=\|(u_1^-)^{u_1^+}\cdot(u_1^+u_2^+)\cdot u_2^-\|_{{\rm EL}_Q}\\
&\leq\|(u_1^-)^{u_1^+}\|_{{\rm EL}_Q}+\|u_1^+u_2^+\|_{{\rm EL}_Q}+\|u_2^-\|_{{\rm EL}_Q}\\
&=\|(u_1^-)\|_{{\rm EL}_Q}+\|u_1^+u_2^+\|_{{\rm EL}_Q}+\|u_2^-\|_{{\rm EL}_Q}=3\|U^+(C_n,R)\|_{{\rm EL}_Q}.
\end{align*}
The last equation follows from the fact, that $U^+(C_n,R)$ and $U^-(C_n,R)$ are conjugate in ${\rm Sp}_{2n}(R).$ Next, according to 
Claim~\ref{claim_upper_triangular_elq}, for each $u\in U^+(C_n,R),$ there is a $u'\in U^+(C_2,R)$ such that $\|u\cdot u'\|_{{\rm EL}_Q}\leq 3(n-2).$ But 
$C_2^+$ only contains four roots and hence $\|u'\|_{{\rm EL}_Q}\leq 4$ and so $\|u\|_{{\rm EL}_Q}\leq 3n-2.$ This finishes the proof.
\end{proof}

We are now able to prove the first part of Theorem~\ref{strong_bound_explicit_semi_local}:

\begin{proof}
Let $S=\{A_1,\dots,A_k\}\subset{\rm Sp}_{2n}(R)$ normally generate ${\rm Sp}_{2n}(R)$. For $1\leq i\leq k$, let $I(A_i)\subset\varepsilon_s(A_i,320n)$ be the ideal given by Theorem~\ref{level_ideal_explicit_Sp_2n} with $V(I(A_i))\subset\Pi(\{A_i\}).$ However, Corollary~\ref{necessary_cond_conj_gen} yields $V(I(A_1)+\cdots+I(A_k))\subset\Pi(S)=\emptyset$ and so no maximal ideal can contain the ideal $I(A_1)+\cdots+I(A_k)$. Thus $\sum_{i=1}^k I(A_i)=R$ and so
\begin{equation}\label{explicit_upper_bounds_semi_local_rings_eq_1}
R=\varepsilon_s(S,320nk).
\end{equation}

According to Corollary~\ref{sum_decomposition_level_ideal_local}, each of the $I(A_i)$ is a sum of $7n$ ideals $J_1(A_i),\dots,J_{7n}(A_i),$ 
each of which is contained in $\varepsilon_s(A_i,64)$. Hence 
\begin{equation*}
\sum_{i=1}^k \sum_{j=1}^{7n}J_j(A_i)=R
\end{equation*}
holds. Next, let $m$ be one of the maximal ideals of $R.$ Clearly not all of the ideals $J_j(A_i)$ can be contained in $m.$ Hence there are 
$i(m)\in\{1,\dots,k\}$ and $j(m)\in\{1,\dots,7n\}$ with
\begin{equation*}
J_{j(m)}(A_{i(m)})\not\subset m.
\end{equation*}
But this implies that 
\begin{equation*}
\sum_{m\text{ maximal ideal in R}} J_{j(m)}(A_{i(m)})
\end{equation*}
cannot be contained in any maximal ideal and thus must be the entire ring $R.$ But this implies
\begin{equation}\label{explicit_upper_bounds_semi_local_rings_eq_2}
R=\varepsilon_s(S,64q)
\end{equation}
Summarizing (\ref{explicit_upper_bounds_semi_local_rings_eq_1}) and (\ref{explicit_upper_bounds_semi_local_rings_eq_2}) yields
\begin{equation}
R=\varepsilon_s(S,64\min\{q,5nk\})
\end{equation} 
holds. However, let $\alpha$ be a short, positive simple root in $C_n$ and $\beta$ a long, positive, simple root in $C_n$, such that the root subsystem of $C_n$ spanned by $\alpha$ and $\beta$ is isomorphic to $C_2.$ Next, we know that $\varepsilon_{\alpha}(x),\varepsilon_{\alpha+\beta}(\pm x)$ are elements of
$B_S(64\min\{q,5nk\})$ and hence
\begin{equation*}
\varepsilon_{2\alpha+\beta}(\pm x)=(\varepsilon_{\alpha}(x),\varepsilon_{\beta}(1))\cdot\varepsilon_{\alpha+\beta}(\mp x) 
\end{equation*} 
is an element of $B_S(192\min\{q,5nk\}).$ Phrased differently, the set $B_S(192\min\{q,5nk\})$ contains all root elements of ${\rm Sp}_{2n}(R).$
But $R$ is semi-local and hence of stable range $1$ by \cite[Lemma~6.4, Corollary~6.5]{MR0174604}. So Corollary~\ref{sp_2n_sl_n_stable_root_elements_principal_ideal_domain} yields $\|{\rm Sp}_{2n}(R)\|_{{\rm EL}_Q}\leq 9n-6.$ This bound together with the fact that $B_S(192\min\{q,5nk\})$ contains all root elements, implies 
\begin{align*}
\|{\rm Sp}_{2n}(R)\|_S&\leq\|{\rm Sp}_{2n}(R)\|_{{\rm EL}_Q}\cdot\|{\rm EL}_Q\|_S\leq (9n-6)\cdot 192\min\{q,5nk\}\\
&=576(3n-2)\min\{q,5nk\}. 
\end{align*}
This finishes the proof.
\end{proof}

\section{Rings of S-algebraic integers}

First, recall the definition of S-algebraic integers:

\begin{mydef}\cite[Chapter~I, \S 11]{MR1697859}\label{S-algebraic_numbers_def}
Let $K$ be a finite field extension of $\mathbb{Q}$. Then let $S$ be a finite subset of the set $V$ of all valuations of $K$ such that $S$ contains all archimedean valuations. Then the ring $\C O_S$ is defined as 
\begin{equation*}
\C O_S:=\{a\in K|\ \forall v\in V-S: v(a)\geq 0\}
\end{equation*}
and $\C O_S$ is called \textit{the ring of $S$-algebraic integers in $K.$} Rings of the form $\C O_S$ are called \textit{rings of S-algebraic integers.} 
\end{mydef}

Remember the word norm $\|\cdot\|_{\rm EL}$ from Definition~\ref{root_elements_word_norms}. Then for $R$ a ring of S-algebraic integers, the group ${\rm Sp}_4(R)$ is boundedly generated by root elements as observed by Tavgen:

\begin{Theorem}\cite{MR1044049}
\label{Tavgen}
Let $K$ be a number field and $R$ a ring of S-algebraic integers in $K.$ Further let
\begin{equation*}
\Delta:=\max\{|\{p|\ p\text{ a prime divisor of }{\rm discr}_{K|\mathbb{Q}}\}|,1\}
\end{equation*}
be given. Then $\|{\rm Sp}_{4}(R)\|_{\rm EL}\leq 180\Delta+27.$ Furthermore, if $R$ is a principal ideal domain or $\Delta=1$, then the bounds can be improved to
$\|{\rm Sp}_{4}(R)\|_{\rm EL}\leq 159.$
\end{Theorem}

\begin{remark}
This is not the bounded generation result as found in \cite{MR1044049}. Instead it is a summary of the result \cite[Corollary~4]{MR1044049} and 
\cite[Proposition~1]{MR1044049} for the first inequality. The second inequality comes from applying possible improvements as appearing in Carter and Keller's paper
\cite{MR704220} in the principal ideal domain-case and the $\Delta=1$-case. 
\end{remark}

Also note the following bounded generation result for ${\rm SL}_2(R)$ by Rapinchuk, Morgan and Sury:

\begin{Theorem}\cite[Theorem~1.1]{MR3892969}\label{Rapinchuck_bounded_generation}
Let $R$ be a ring of S-algebraic integers with infinitely many units. Then $\|{\rm SL}_2(R)\|_{{\rm EL}}\leq 9.$
\end{Theorem}

We can prove the first part of Theorem~\ref{strong_bound_explicit_alg_integer} now:

\begin{proof}
Let $S=\{A_1,\dots,A_k\}$ be a normal generating set of ${\rm Sp}_{2n}(R).$ The proof proceeds in two steps: We first show that all root elements 
of ${\rm Sp}_{2n}(R)$ are contained in the ball $B_S(960n|S|).$ This implies that ${\rm EL}_Q=\{A\varepsilon_{\phi}(x)A^{-1}|x\in R,\phi\in C_n,A\in{\rm Sp}_{2n}(R)\}$
is a subset of $B_S(960n|S|).$ Second, we show that $\|{\rm Sp}_{2n}(R)\|_{{\rm EL}_Q}\leq 12n+\Delta(R)$. This two steps then imply the first part of 
Theorem~\ref{strong_bound_explicit_alg_integer} as follows:
\begin{align*}
\|{\rm Sp}_{2n}(R)\|_S\leq \|{\rm EL}_Q\|_S\cdot\|{\rm Sp}_{2n}(R)\|_{{\rm EL}_Q}\leq 960n|S|\cdot(12n+\Delta(R)).
\end{align*}
For the first step, note that Theorem~\ref{level_ideal_explicit_Sp_2n} implies that there are ideals $I(A_1),\dots,I(A_k)$ such that for all $i=1,\dots,k$, one has
$I(A_i)\subset\varepsilon_s(A_i,320n)$ and $V(I(A_i))\subset\Pi(\{A_i\}).$ But this implies first that $I_S:=I(A_1)+\cdots+I(A_k)\subset\varepsilon_s(S,320n\cdot|S|)$
and second using Lemma~\ref{intersection_v_Pi}:
\begin{equation*}
V(I_S)=V(I(A_1))\cap\dots\cap V(I(A_k))\subset\Pi(A_1)\cap\dots\cap\Pi(A_k)=\Pi(S).
\end{equation*}
However, we know $\Pi(S)=\emptyset$ from Corollary~\ref{necessary_cond_conj_gen} and so $I_S$ is not contained in any maximal ideal of $R$. Thus $I_S=R$ and 
$\varepsilon_s(S,320n\cdot|S|)=R.$ Then proceeding as in the proof of the first part of Theorem~\ref{strong_bound_explicit_semi_local}, one obtains 
that $B_S(960n\cdot|S|)$ contains all root elements of ${\rm Sp}_{2n}(R)$ and hence the first step of the proof is finished. 

For the second step, we first give upper bounds on $\|{\rm Sp}_4(R)\|_{{\rm EL}_Q}$ depending on $R$. First, if $R$ is a quadratic imaginary ring of integers or $\mathbb{Z}$, we have $\|{\rm Sp}_4(R)\|_{{\rm EL}_Q}\leq\|{\rm Sp}_4(R)\|_{EL}\leq 159$ according to Theorem~\ref{Tavgen}.

On the other hand, if $R$ is not a ring of quadratic imaginary integers or $\mathbb{Z}$, then $R$ has infinitely many units according to \cite[Corollary~11.7]{MR1697859}. This implies $\|{\rm SL}_2(R)\|_{EL}\leq 9$ for those rings by Theorem~\ref{Rapinchuck_bounded_generation}. 
But rings of algebraic integers have stable range at most $2$ and so according to Proposition~\ref{rank_1_boundedness}, this implies 
\begin{equation*}
{\rm Sp}_4(R)=(U^+(C_2,R)U^-(C_2,R))^4U^+(C_2,R)\text{ or }{\rm Sp}_4(R)=U^-(C_2,R)(U^+(C_2,R)U^-(C_2,R))^4.
\end{equation*}
But $C_2$ has four positive roots and hence $\|{\rm Sp}_4(R)\|_{{\rm EL}_Q}\leq\|{\rm Sp}_4(R)\|_{EL}\leq 4\cdot 9=36$ holds. Hence setting 
\begin{equation*}
\Delta'(R):=
\begin{cases} 
 &\text{159, if }R\text{ is a quadratic imaginary ring of integers or }\mathbb{Z}\\
 &\text{36, if }R\text{ is neither of the above}
\end{cases}
\end{equation*}
implies $\|{\rm Sp}_4(R)\|_{{\rm EL}_Q}\leq\Delta'(R)$ for all rings of S-algebraic integers with class number $1.$
Proposition~\ref{sp_2n_sl_n_stable_root_elements_principal_ideal_domain} then implies
\begin{equation*}
\|{\rm Sp}_{2n}(R)\|_{{\rm EL}_Q}\leq 12(n-2)+\|{\rm Sp}_4(R)\|_{{\rm EL}_Q}\leq 12(n-2)+\Delta'(R)=12n+\Delta(R).
\end{equation*}
This finishes the second step and the proof.
\end{proof}

\begin{remark}
One could improve the upper bounds on $\Delta_k({\rm Sp}_{2n}(R))$ in Theorem~\ref{strong_bound_explicit_alg_integer} and Theorem~\ref{strong_bound_explicit_semi_local} by a factor of $1.5$ by using that short root elements have better upper bounds with respect to $\|\cdot\|_S$ than long ones. However, this would require a more cumbersome argument. 
\end{remark}

\section{Lower bounds on $\Delta_k({\rm Sp}_{2n}(R))$} 

First, we need the following:

\begin{Proposition}
\label{dimension_counting_sln_sp_2n}
Let $K$ be a field, $t\in K-\{0\}, n\geq 2$ and $\phi\in C_n$ long. Then the element $E:=\varepsilon_{\phi}(t)$ normally generates ${\rm Sp}_{2n}(K)$ and 
$\|{\rm Sp}_{2n}(K)\|_E\geq 2n.$ 
\end{Proposition}

\begin{proof}
First, we show that $E$ indeed normally generates ${\rm Sp}_{2n}(K).$ To this end, first note that it is well-known that the group ${\rm Sp}_{2n}(K)$ is generated by its root elements. Hence according to Corollary~\ref{necessary_cond_conj_gen}, the element $E$ normally generates ${\rm Sp}_{2n}(R),$ if $\Pi(\{E\})=\emptyset.$ However, the field $K$ has only one maximal ideal, namely $(0),$ and so $\Pi(\{E\})$ can only be non-empty, if $E$ is trivial, which is not the case.
Next, using the conventions from Section~\ref{section_matrix_calculations_sp_2n}, we can (possibly after conjugation with Weyl group elements) assume 
$E=I_{2n}+te_{1,n+1}.$ We define the subspace $I(l):=\{v\in K^{2n}|l(v)=v\}$ for a linear map $l:K^{2n}\to K^{2n}.$ We prove next that for $l_1,l_2:K^{2n}\to K^{2n}$, one has
\begin{equation}
\label{fix_space_dim_formula}
{\rm dim}_K(I(l_1l_2))\geq{\rm dim}_K(I(l_1))+{\rm dim}_K(I(l_2))-2n.
\end{equation}
To see this, observe first that $I(l_1)\cap I(l_2)\subset I(l_1l_2)$ and hence 
\begin{align*}
{\rm dim}_K(I(l_1l_2))&\geq{\rm dim}_K(I(l_1)\cap I(l_2))={\rm dim}_K(I(l_1))+{\rm dim}_K(I(l_2))-{\rm dim}_K(\langle I(l_1),I(l_2)\rangle)\\
&\geq {\rm dim}_K(I(l_1))+{\rm dim}_K(I(l_2))-2n.
\end{align*}
 
Observe that the linear map $E:K^{2n}\to K^{2n}$ induced by $E$ has 
\begin{equation*}
I(E^{-1})=I(E)=Ke_1\oplus\cdots\oplus Ke_n\oplus Ke_{n+2}\oplus\cdots\oplus Ke_{2n}.
\end{equation*}
Hence ${\rm dim}_K I(E)=2n-1={\rm dim}_K I(E^{-1})$ holds. Note further for $X\in K^{2n\times 2n}$, 
$A\in {\rm GL}_{2n}(K)$ and $v\in K^{2n}$, that the following holds:
\begin{equation*}
v\in I(X)\text{ precisely if }Av\in I(AXA^{-1}).
\end{equation*}
Hence $I(AXA^{-1})=AI(X)$ holds and thus ${\rm dim}_K I(X)={\rm dim}_K I(AXA^{-1}).$ Hence for each conjugate $X$ of $E$ or $E^{-1}$ in ${\rm Sp}_{2n}(K)$, one has ${\rm dim}_K(I(X))=2n-1.$ Next, let $X_1,\dots,X_k$ be either conjugates of $E$ or $E^{-1}$ in ${\rm Sp}_{2n}(K)$ or $I_{2n}.$ Then ${\rm dim}_K(I(X_1\cdots X_k))\geq 2n-k$ follows by induction on $k\in\mathbb{N}$ from (\ref{fix_space_dim_formula}). This implies in particular that for each $A\in B_E(2n-1)$ there is a non-trivial vector $v(A)\in K^{2n}$ fixed by $A.$ Hence each element of $B_E(2n-1)$ has eigenvalue $1.$ So if $\|{\rm Sp}_{2n}(K)\|_E\leq 2n-1$ or equivalently $B_E(2n-1)={\rm Sp}_{2n}(K)$ were to hold, then each element $A\in{\rm Sp}_{2n}(K)$ would have eigenvalue $1.$ Thus it suffices to give an element $A\in {\rm Sp}_{2n}(K)$ without the eigenvalue $1$ to finish the proof. To this end, observe that for $B\in{\rm SL}_n(K)$, the matrix
\begin{align*}
A=\left(
\begin{array}{c|c}
B & 0_n\\
\midrule
0_n & B^{-T}
\end{array}
\right)
\end{align*}
is an element of ${\rm Sp}_{2n}(R)$ with characteristic polynomial 
\begin{equation*}
\chi_A(x)=\chi_B(x)\chi_{B^{-T}}(x)=\chi_B(x)\chi_{B^{-1}}(x). 
\end{equation*}
But this implies that $A$ has eigenvalue $1$ precisely if either $B$ or $B^{-1}$ has eigenvalue $1.$ Yet $B^{-1}$ has eigenvalue $1$ precisely if $B$ does. Thus it suffices to provide an element $B\in{\rm SL}_n(K)$ without eigenvalue $1$ to finish the proof, but such matrices $B$ clearly exist.
\end{proof}

\begin{remark}
This `dimension counting'-strategy is quite well-known and was mentioned to me by B. Karlhofer in a different context, but it is also alluded to in Lawther's and Liebeck's paper \cite[p.~120]{lawther1998diameter}.
\end{remark}

We can finish the proof of Theorem~\ref{strong_bound_explicit_semi_local} now by providing lower bounds on $\Delta_k({\rm Sp}_{2n}(R))$:

\begin{proof}
Let $q<+\infty$ be the number of maximal ideals in $R.$ Note that $\Delta_{k}({\rm Sp}_{2n}(R))\leq\Delta_{k+1}({\rm Sp}_{2n}(R))$ holds for all $k\in\mathbb{N}.$ Thus we can restrict ourselves to the case of $k\leq q.$ Next, let $\C M_1,\dots,\C M_q$ be the maximal ideals in $R.$ Note that as maximal ideals $\C M_i$ and $\C M_j$ are coprime for $1\leq i\neq j\leq q$. Hence we obtain using the Chinese Remainder Theorem that
\begin{equation*}
p:R\to\prod_{i=1}^q R/\C M_i,x\mapsto (x+\C M_1,\dots,x+\C M_q)
\end{equation*}
is an epimorphism. Thus we can pick elements $x_1,\dots,x_k\in R$ such that
\begin{equation*}
p(x_j)=(0+\C M_1,\dots,0+\C M_{j-1},1+\C M_j,0+\C M_{j+1},\dots,0+\C M_k,1+\C M_{k+1},\dots,1+\C M_q).
\end{equation*}
holds for all $1\leq j\leq k.$ Next, let $\phi$ be a long root in $C_n$ and set $S:=\{\varepsilon_{\phi}(x_1),\dots,\varepsilon_{\phi}(x_k)\}.$ We finish the proof now by showing two claims:
First, that $S$ is a normally generating subset of ${\rm Sp}_{2n}(R)$ and second that $\|{\rm Sp}_{2n}(R)\|_S\geq 2nk$ holds.

To show the first claim, note that due to the choice of the $x_1,\dots,x_k$ and the fact that the $\C M_1,\dots,\C M_q$ are all the maximal ideals of $R$, we obtain for $1\leq j\leq k$ that:
\begin{equation*}
\Pi(\{\varepsilon_{\phi}(x_j)\})=\{\C M_i|1\leq i\neq j\leq k\}
\end{equation*}
and hence $\Pi(S)=\emptyset.$ But then Corollary~\ref{necessary_cond_conj_gen} implies that $S$ is indeed a normal generating set of ${\rm Sp}_{2n}(R).$
Next, consider the map 
\begin{equation*}
\pi:{\rm Sp}_{2n}(R)\to\prod_{i=1}^k{\rm Sp}_{2n}(K_i),A\mapsto(\pi_{\C M_1}(A),\dots,\pi_{\C M_k}(A))
\end{equation*}
for the fields $K_i$ defined by $K_i:=R/\C M_i$ and the $\pi_{\C M_i}:{\rm Sp}_{2n}(R)\to{\rm Sp}_{2n}(K_i)$ being the corresponding reduction homomorphisms.
Then note that 
\begin{align*}
\pi(\varepsilon_{\phi}(x_i))&=
(\varepsilon_{\phi}(x_i+\C M_1),\dots,\varepsilon_{\phi}(x_i+\C M_{i-1}),\varepsilon_{\phi}(x_i+\C M_i),\varepsilon_{\phi}(x_i+\C M_{i+1}),\dots,\varepsilon_{\phi}(x_i+\C M_k))\\
&=(1,\dots,1,\varepsilon_{\phi}(1+\C M_i),1,\dots,1).
\end{align*}
Thus the only non-trivial component of $\pi(\varepsilon_{\phi}(x_i))$ is the ${\rm Sp}_{2n}(K_i)$-component equal to $\varepsilon_{\phi}(x_i+\C M_i)$ and 
${\rm Sp}_{2n}(K_i)$ is normally generated by $\varepsilon_{\phi}(x_i+\C M_i)$. Also this implies that the only non-trivial component of any conjugate of 
$\pi(\varepsilon_{\phi}(x_i))$ is the ${\rm Sp}_{2n}(K_i)$-component. Together this implies that $\pi(S)$ normally generates $\prod_{i=1}^k {\rm Sp}_{2n}(K_i)$ and  
\begin{equation*}
\|{\rm Sp}_{2n}(R)\|_S\geq \|\prod_{i=1}^k {\rm Sp}_{2n}(K_i)\|_{\pi(S)}=\sum_{i=1}^k\|{\rm Sp}_{2n}(K_i)\|_{\varepsilon_{\phi}(x_i+\C M_i)}.
\end{equation*}
Thus to finish the proof, it suffices to now apply Proposition~\ref{dimension_counting_sln_sp_2n} to obtain 
\begin{equation*}
\|{\rm Sp}_{2n}(K_i)\|_{\varepsilon_{\phi}(x_i+\C M_i)}\geq 2n
\end{equation*}
for all $i=1,\dots,k$. 
\end{proof}

The proof for the second part of Theorem~\ref{strong_bound_explicit_alg_integer} works the same way. The only difference is that the $x_1,\dots,x_k$ are instead chosen as $x_i:=p_1\cdots\hat{p_i}\cdots p_k$ for $p_1,\dots,p_k$ the generators of $k$ distinct maximal ideals and the hat denoting the omission of the corresponding prime factory. 

\section*{Closing remarks}

In the course of this paper, we assumed that the ring $R$ is a principal ideal domain. This however is mainly a method to simplify the calculations. Instead one can also use that both semi-local rings and rings of algebraic integers have stable range at most $2$ and then argue similar to the proof of 
Bass' \cite[Theorem~4.2(e)]{MR0174604}. However, this requires further intermediate steps and does not change the overall asymptotic of the bounds on  
$\Delta_k({\rm Sp}_{2n}(R))$ in $k$ and $n$, but merely the appearing coefficients. 

Kedra, Libman and Martin \cite{KLM} have shown that $\Delta_k({\rm SL}_n(R))$ for $n\geq 3$ also has an upper bound proportional to $n^2k$ for $R$ a ring of S-algebraic integers with class number $1$. It is clear that similar arguments as in the present paper will also work for all S-arithmetic Chevalley groups $G(\Phi,R)$ for $\Phi$ an irreducible root system of rank at least $3$ to show that
\begin{equation*}
{\rm rank}(\Phi)\cdot k\lesssim\Delta_k(G(\Phi,R))\lesssim{\rm rank}(\Phi)^2\cdot k.
\end{equation*}
This raises the question regarding the true asymptotic of $\Delta_k(G(\Phi,R))$. At the time of writing, we were not aware of results that would help decide this. However, the anomynous referee for \cite{General_strong_bound} suggested a strategy that would among else, imply that the asymptotic of $\Delta_k(G(\Phi,R))$ agrees with the lower linear bound, if for each irreducible root system $\Phi$ of rank at least $3$ and each $I$ a non-trivial principal ideal domain in $R$, the following conjecture holds:

\begin{Conjecture}\label{fundamental_conjecture}
Let $X_{I,\Phi}:=\{A\varepsilon_{\phi}(x)A^{-1}|x\in I,A\in G(\Phi,R),\phi\in\Phi\}$ be given and let $N_{I,\Phi}$ be the subgroup of $G(\Phi,R)$ generated by $X_{I,\Phi}.$ Then there is a constant $K:=K(I,\Phi)$ proportional to ${\rm rank}(\Phi)$ and independent of $R$ and $I$ such that $N_{I,\Phi}=X_{I,\Phi}^K$ holds.
\end{Conjecture}

There is only one such result known to us, a theorem by Morris \cite[Theorem~6.1(1)]{MR2357719} implying that there is such a $K(I,A_n)$ for $n\geq 3$, but this $K(I,A_n)$ depends on the cardinality of $R/I,$ the number $n$ and the degree $[K:\mathbb{Q}]$ of the number field $K$ containing $R.$ 
It is easy to see that a constant $K(I,A_n)$ as required must depend on $n$ (or equivalently ${\rm rank}(A_n)$), but whether such a constant $K(I,A_n)$ must depend on $I$, $|R/I|$ and $[K:\mathbb{Q}]$ is somewhat unclear.
In an upcoming paper about bounds for $\Delta_k({\rm Sp}_4(R))$, we will deal with a relatively simple special case for $I=2R$ and $\Phi=C_2$ to show that at least in certain cases $K(I,\Phi)$ can be made to not depend on $R$ and the field extension $K|\mathbb{Q}$.

\bibliography{bibliography}
\bibliographystyle{plain}

\end{document}